\documentclass[12pt,leqno]{article}  
\usepackage{amsmath,amssymb,bbm}
\usepackage{amsthm,array,mathdots}
\usepackage{titlesec} 
\usepackage[all,cmtip]{xy}
\usepackage{blkarray,arydshln} 

\titlespacing{\section}{0cm}{3.5pc}{1.5pc}

\makeatletter
\def\@citex[#1]#2{\if@filesw\immediate\write\@auxout{\string\citation{#2}}\fi
  \def\@citea{}\@cite{\@for\@citeb:=#2\do
    {\@citea\def\@citea{\@citesep}\@ifundefined
       {b@\@citeb}{{\bf ?}\@warning
       {Citation `\@citeb' on page \thepage \space undefined}}%
{\csname b@\@citeb\endcsname}}}{#1}}
\def\@citesep{; }
\makeatother

\newtheoremstyle{Kang}{}{}{\itshape}{}{\bf}{}{.5em}{}
\theoremstyle{Kang}
\newtheorem{theorem}{Theorem}[section]

\newtheorem{lemma}[theorem]{Lemma}

\newtheorem{prop}[theorem]{Proposition}

\newtheoremstyle{Kremark}{}{}{}{}{\bf}{}{.5em}{}
\theoremstyle{Kremark}
\newtheorem*{remark}{Remark.}
\newtheorem{defn}[theorem]{Definition}

\newtheorem{other}{}

\newenvironment{Case}[1]{\medskip {\it Case #1.}}{}

\allowdisplaybreaks[1]  

\def\fn#1{\operatorname{#1}} 
\def\bm#1{\mathbbm{#1}}

\title{Rationality Problems for Relation Modules of Dihedral Groups}
\author{Akinari Hoshi$^{(1)}$, Ming-chang Kang$^{(2)}$ and Aiichi Yamasaki$^{(3)}$ \\[3mm]
\begin{minipage}{16cm} \begin{description} \itemsep=-1pt
\item[] $^{(1)}$Department of Mathematics, Niigata University,
Niigata 950-2181,\\ Japan, E-mail: hoshi@math.sc.niigata-u.ac.jp \item[]
$^{(2)}$Department of Mathematics, National Taiwan University,
Taipei,\\ Taiwan, E-mail: kang@math.ntu.edu.tw \item[]
$^{(3)}$Department of Mathematics, Kyoto University, Kyoto 606-8502,\\
Japan, E-mail: aiichi.yamasaki@gmail.com
\end{description} \end{minipage}}
\date{}

\begin{document}

\maketitle

\footnote{\textit{\!\!\! $2010$ Mathematics Subject
Classification}. 14E08, 11R33, 20C10, 20F05.}
\footnote{\textit{\!\!\! Keywords and phrases}. Rationality
problem, integral representation, free presentation, relation
module, algebraic torus.} \footnote{\!\!\! This work was partially
supported by JSPS KAKENHI Grant Numbers 24540019, 25400027, 16K05059.}

\footnote{\!\!\! Parts of the work were finished when the
first-named author and the third-named author were visiting the
National Center for Theoretic Sciences (Taipei Office), whose
support is gratefully acknowledged.}

\begin{abstract}
{\noindent\bf Abstract.} Let $D_n=\langle
\sigma,\tau:\sigma^n=\tau^2=1,\tau\sigma\tau^{-1}=\sigma^{-1}\rangle$
be the dihedral group of order $2n$ where $n\ge 2$. Let $1\to R\to
F\xrightarrow{\varepsilon} D_n\to 1$ be the free presentation of
$D_n$ where $F=\langle s_1,s_2\rangle$ is the free group of rank
two and $\varepsilon(s_1)=\sigma$, $\varepsilon(s_2)=\tau$. The
conjugation maps of $F$ provide an action of $F$ on $R$. Thus
$R^{ab}:=R/[R,R]$ becomes a $\bm{Z}[D_n]$-lattice. $R^{ab}$ is
called the relation module of $D_n$ by Gruenberg. For any Galois extension $K/k$ with $\fn{Gal}(K/k)=D_n$, $K(R^{ab})^{D_n}$ is stably rational over $k$ if and only if $n$ is an odd integer. Moreover, (i) if $K(R^{ab})^{D_n}$ is stably rational over $k$, it is rational over $k$, and (ii) if $K(R^{ab})^{D_n}$ is not stably rational over $k$ and $k$ is an infinite field, then $K(R^{ab})^{D_n}$ is not retract rational over $k$. We will also discuss the rationality of the purely monomial action where $D_n$ acts trivially on the field $k$. It will be shown that, if $n$ is odd, then $k(R^{ab})^{D_n}=k(D_n)(t)$ for any field $k$. \end{abstract}

\newpage
\section{Introduction}

Let $G$ be a finite group, $F=\langle s_1,\ldots,s_d\rangle$ be the free group of rank $d$.
A surjective homomorphism $\varepsilon:F\to G$ gives rise to a free presentation $1\to R \to F \xrightarrow{\varepsilon} G\to 1$ of $G$.
Since $R$ is a normal subgroup of $F$, $F$ acts on $R$ by the conjugation maps.
Consequently $G$ acts on $R^{ab}:=R/[R,R]$ where $[R,R]$ is the commutator subgroup of $R$.
Since $R$ itself is a free group of rank $1+(d-1)|G|$ by Schreier's Theorem \cite[page 36]{Kur},
$R^{ab}$ is a free abelian group of the same rank.
It follows that $R^{ab}$ becomes a $\bm{Z}[G]$-module which is a free abelian of finite rank,
i.e.\ $R^{ab}$ is a $\bm{Z}[G]$-lattice (in short, $G$-lattice) in the sense of \cite[page 524]{CR1} and Section 2 of this article.
In other words, $R^{ab}$ provides an integral representation $G\to GL_t(\bm{Z})$ for some positive integer $t$ (see \cite[Chapter 3]{CR1}).
The $G$-lattice $R^{ab}$ is called the relation module of $G$ by Gruenberg \cite{Gr1,Gr2}.

Consider the case $G=D_n=\langle \sigma,\tau:\sigma^n=\tau^2=1,
\tau\sigma \tau^{-1}=\sigma^{-1}\rangle$, the dihedral group of
order $2n$ where $n\ge 2$.

\begin{defn} \label{d1.1}
Let $1\to R\to F\xrightarrow{\varepsilon} D_n\to 1$ be a free
presentation of $D_n$ where $F=\langle s_1,s_2\rangle$ the free
group of rank 2, and $\varepsilon(s_1)=\sigma$,
$\varepsilon(s_2)=\tau$. The $D_n$-lattice $R^{ab}$ is the subject
we will study in Section 5. The relation module $R^{ab}$ depends
on the group $G$, the free group $F$, and also on the epimorphism
$\varepsilon:F\to D_n$ (see \cite[page 79]{Gr2} for details).

\end{defn}

The following theorem is due to Gruenberg
and Roggenkamp.

\begin{theorem} [Gruenberg and Roggenkamp \cite{GR}] \label{t1.2}
Let $1\to R\to F\xrightarrow{\varepsilon} D_n\to 1$ be the free
presentation of $D_n$ in Definition \ref{d1.1} where $n \ge 2$. The lattice $R^{ab}$ is an
indecomposable $D_n$-lattice if and only if $n$ is an even integer.
\end{theorem}

In this article we focus on the rationality problem for $R^{ab}$. If $k$ is a
field and $M$ is a $G$-lattice, we will define a multiplicative
action of $G$ on the field $k(M)$ in Definition \ref{d2.3}. Thus
we may consider the rationality problem of $k(R^{ab})^{D_n}$ in
Theorem \ref{t1.3}. Similarly we may consider the rationality
problem of $K(R^{ab})^{D_n}$ if $K/k$ is a Galois extension with
$\fn{Gal}(K/k)=D_n$. The reader may find a description of
Noether's problem and the definition of $k(D_n)$ of Theorem
\ref{t1.3} in Definition \ref{d6.1}. As to the terminology of
$k$-rationality, stable $k$-rationality, retract $k$-rationality,
see Definition \ref{d2.4}. Here are the main results of this paper.

\begin{theorem} \label{t1.4}
Let $R^{ab}$ be the relation module associated to the free resolution $1\to R\to F\xrightarrow{\varepsilon} D_n\to 1$ in Definition \ref{d1.1}. Let $K/k$ be a Galois
extension with $\fn{Gal}(K/k)=D_n$. If
$n$ is an odd integer $\ge 3$, then $K(R^{ab})^{D_n}$ is rational
over $k$. If $n$ is an even integer $\ge
2$, then $K(R^{ab})^{D_n}$ is not stably rational over $k$; moreover, if $k$ is an infinite field, then $K(R^{ab})^{D_n}$ is not retract rational over $k$.
\end{theorem}

\begin{theorem} \label{t1.3}
Let $R^{ab}$ be the relation module associated to the free resolution $1\to R\to F\xrightarrow{\varepsilon} D_n\to 1$ in Definition \ref{d1.1}. If $n$ is
an odd integer $\ge 3$ and $k$ is a field, then $k(R^{ab})^{D_n}=k(D_n)(t)$ for some
element $t$ transcendental over the field $k(D_n)$. Consequently,
if Noether's problem for $D_n$ over the field $k$ has an
affirmative answer (e.g.\ $\zeta_n+\zeta_n^{-1}\in k$ where
$\zeta_n$ is a primitive $n$-th root of unity), then
$k(R^{ab})^{D_n}$ is rational over $k$.
\end{theorem}

Note that the stable rationality in Theorem \ref{t1.4} follows from Endo and Miyata's Theorem (see Theorem \ref{t2.7} and Theorem \ref{t5.8}). However, to assert the stronger result that $K(R^{ab})^{D_n}$ is rational over $k$ if $n$ is odd, extra efforts are required. First we need an explicit decomposition $R^{ab}\simeq M_+ \oplus \widetilde{M}_+$ (see Theorem \ref{t5.7}) which was anticipated by Theorem \ref{t1.2}; see Definition \ref{d3.1} and Definition \ref{d3.3} for the constructions of $M_+$ and $\widetilde{M}_+$. Then applying Theorem \ref{t3.4} (i.e. the identity $\widetilde{M}_+ \oplus \bm{Z}\simeq
\bm{Z}[D_n/\langle \sigma \rangle] \oplus \bm{Z}[D_n/\langle
\tau\rangle]$) we can establish the rationality of $K(R^{ab})^{D_n}$. We remark also that the case for $n=2$ or $3$ in Theorem \ref{t1.4}, were proved by Kunyavskii \cite{Ku1,Ku2}.

So far as we know, the rationality problem of $k(R^{ab})^{D_n}$
was investigated first in a paper of Snider. His result is the
following.

\begin{theorem}[Snider \cite{Sn}] \label{t1.5}
Let the notations be the same as in Theorem \ref{t1.3}.
Let $k$ be any field.
If $n=2$, then $k(R^{ab})^{D_2}$ is rational over $k$.
If $n$ is an odd integer $\ge 3$ and $\zeta_n\in k$,
then $k(R^{ab})^{D_n}$ is stably rational over $k$.
\end{theorem}

An alternative proof of Snider's Theorem for $k(R^{ab})^{D_2}$ will be given in Section 5 using an idea of Kunyavskii \cite{Ku1}. In fact, we will prove in Theorem \ref{t6.9} that $k(M)^{D_2}$ is stably rational over $k$ for any $D_2$-lattice $M$ of rank $\le 5$. We don't know the answer if $M$ is a $D_2$-lattice $M$ of rank $\ge 6$. Nor do we know the answer to the stable rationality of $k(R^{ab})^{D_n}$ if $n \ge 4$ is an even integer.

The motivation of Snider to consider the above rationality problem
came from the attempt to prove whether the norm residue morphism
$R_{n,F}:K_2(F)/nK_2(F)\to H^2(F,\mu_n^{\otimes 2})$ is an
isomorphism where $F$ is a field \cite[page 144]{Mi}. For the
surjectivity of $R_{n,F}$, Snider used the idea of generic crossed
products with group $G$ ($G$ varies over all the possible finite
groups); in particular, he considered the generic crossed products
with group $D_n$. Thus the rationality of $k(R^{ab})^{D_n}$
enables him to prove the surjectivity of $R_{n,F}$ if the
corresponding Brauer class is similar to a crossed product with
group $D_n$. The isomorphism of $R_{n,F}$ was solved by Merkurjev
and Suslin \cite{MS} in 1983. However, the rationality of
$k(R^{ab})^{D_n}$ or $k(R^{ab})^G$ (where $G$ is any finite group)
remains a challenging problem.

We organize this paper as follows. In Section 3, we construct various $D_n$-lattices when $n$ is an odd integer.
These $D_n$-lattices turn out to be fundamental building blocks of
indecomposable $D_p$-lattices when $p$ is an odd prime number (see the
paper of Myrma Pike Lee \cite{Le} for details). We will prove some identities of these
$D_n$-lattices, e.g.\ $\widetilde{M}_+ \oplus \bm{Z}\simeq
\bm{Z}[D_n/\langle \sigma \rangle] \oplus \bm{Z}[D_n/\langle
\tau\rangle]$ (see Theorem \ref{t3.4} and Theorem \ref{t3.5}). Besides applications in the rationality problem, these identities exemplify the phenomenon that the
Krull-Schmidt-Azumaya's Theorem \cite[page 128]{CR1} may fail in
the category of $G$-lattices. We study some homological properties of
these lattices in Section 4, e.g.\ flabby, coflabby, invertible,
flabby resolutions (see Definition \ref{d2.1}). These homological
properties are used in the proof of Theorem \ref{t5.8}. The proofs of Theorem
\ref{t1.4} and Theorem \ref{t1.3} are given in Section 5.

\bigskip
Terminology and notations. In this paper, all the groups $G$ are
finite groups. A free group $F$ is called a free group of rank $d$
if $F$ is a free group on $d$ generators $s_1,\ldots,s_d$. A free
presentation of $G$, $1\to R\to F\xrightarrow{\pi} G\to 1$, is a
surjective morphism $\pi:F\to G$ where $F$ is a free group and $R$
is the kernel of $\pi$. $C_n$ and $D_n$ denote the cyclic group of
order $n$ and the dihedral group of order $2n$. If $R$ is a group,
then $[R,R]$ is its commutator subgroup and $R^{ab}$ denotes the
quotient group $R/[R,R]$. When $k$ is a field, we write
$\fn{gcd}\{n,\fn{char}k\}=1$ to denote the situation either
$\fn{char}k=0$ or $\fn{char}k=p>0$ with $p\nmid n$. By $\zeta_n$
we mean a primitive $n$-th root of unity in some field. If $M$ is
a module over some ring $A$, we denote by $M^{(n)}$ the direct sum
of $n$ copies of $M$.

\bigskip
Acknowledgments. We thank the anonymous referee for his insightful, constructive and meticulous comments.

\section{\boldmath Preliminaries of $G$-lattices}

Let $G$ be a finite group. Recall that a finitely generated
$\bm{Z}[G]$-module $M$ is called a $G$-lattice if it is
torsion-free as an abelian group. We define the rank of $M$:
$\fn{rank}_{\bm{Z}}M=n$ if $M$ is a free abelian group of rank
$n$. If $M$ is a $G$-lattice, define
$M^0:=\fn{Hom}_{\bm{Z}}(M,\bm{Z})$ which is also a $G$-lattice;
$M^0$ is called the dual lattice of $M$.

A $G$-lattice $M$ is called a permutation lattice if $M$ has a
$\bm{Z}$-basis permuted by $G$. A $G$-lattice $M$ is called stably
permutation if $M\oplus P$ is a permutation lattice where $P$ is
some permutation lattice. $M$ is called an invertible lattice if
it is a direct summand of some permutation lattice. A $G$-lattice
$M$ is called a flabby lattice if $H^{-1}(S,M)=0$ for any subgroup
$S$ of $G$; it is called coflabby if $H^1(S,M)=0$ for any subgroup
$S$ of $G$ (here $H^{-1}(S,M)$ and $H^1(S,M)$ denote the Tate cohomology groups of $S$). If $S$ is a subgroup of $G$, we will write
$\bm{Z}[G/S]$ for the permutation lattice $\bm{Z}[G]
\otimes_{\bm{Z}[S]}\bm{Z}$ where $\bm{Z}$ is the trivial
$\bm{Z}[S]$-lattice. For the basic notions of $G$-lattices, see
\cite{Sw3,Lo}.

\medskip
Let $G$ be a finite group.
Two $G$-lattices $M_1$ and $M_2$ are similar,
denoted by $M_1\sim M_2$, if $M_1\oplus P_1 \simeq M_2 \oplus P_2$ for some permutation $G$-lattices $P_1$ and $P_2$.
The flabby class monoid $F_G$ is the class of all flabby $G$-lattices under the similarity relation.
In particular, if $M$ is a flabby lattice,
$[M]\in F_G$ denotes the equivalence class containing $M$;
we define $[M_1]+[M_2]=[M_1\oplus M_2]$ and thus $F_G$ becomes an abelian monoid \cite{Sw3}.

\begin{defn} \label{d2.1}
Let $G$ be a finite group, $M$ be any $G$-lattice. Then $M$ has a
flabby resolution, i.e. there is an exact sequence of
$G$-lattices: $0\to M\to P\to E\to 0$ where $P$ is a permutation
lattice and $E$ is a flabby lattice. The class $[E]\in F_G$ is
uniquely determined by the lattice $M$ \cite{Sw3}. We define
$[M]^{fl}=[E]\in F_G$, following the nomenclature in \cite[page
38]{Lo}. Sometimes we will say that $[M]^{fl}$ is permutation or
invertible if the class $[E]$ contains a permutation or invertible
lattice.
\end{defn}

\begin{defn} \label{d2.2}
Let $K/k$ be a finite Galois field extension with $G=\fn{Gal}(K/k)$.
Let $M=\bigoplus_{1\le i\le n} \bm{Z}\cdot e_i$ be a $G$-lattice.
We define an action of $G$ on $K(M)=K(x_1,\ldots,x_n)$,
the rational function field of $n$ variables over $K$,
by $\sigma\cdot x_j=\prod_{1\le i\le n} x_i^{a_{ij}}$ if $\sigma\cdot e_j=\sum_{1\le i\le n} a_{ij} e_i \in M$,
for any $\sigma \in G$ (note that $G$ acts on $K$ also).
The fixed field is denoted by $K(M)^G$.
\end{defn}

If $T$ is an algebraic torus over $k$ satisfying
$T\times_{\fn{Spec}(k)}\fn{Spec}(K)\simeq\bm{G}_{m,K}^n$ where
$\bm{G}_{m,K}$ is the one-dimensional multiplicative group over
$K$, then $M:=\fn{Hom}(T,\bm{G}_{m,K})$ is a $G$-lattice and the
function field of $T$ over $k$ is isomorphic to $K(M)^G$ by Galois
descent \cite[page 36; Vo]{Sw3}.

\begin{defn} \label{d2.3}
We give a generalization of $K(M)^G$ in Definition \ref{d2.2}. Let
$M=\bigoplus_{1\le i\le n} \bm{Z}\cdot e_i$ be a $G$-lattice,
$k'/k$ be a finite Galois extension field such that there is a
surjection $G\to \fn{Gal}(k'/k)$. Thus $G$ acts naturally on $k'$
by $k$-automorphisms. We define an action of $G$ on
$k'(M)=k'(x_1,\ldots,x_n)$ in a similar way as $K(M)$. The fixed
field is denoted by $k'(M)^G$. The action of $G$ on $k'(M)$ is
called a purely quasi-monomial action in \cite[Definition
1.1]{HKK}; it is possible that $G$ acts faithfully on $k'$ (the
case $k'=K$) or trivially on $k'$ (the case $k'=k$).
\end{defn}

We recall the notions of rationality, stable rationality and retract rationality.

\begin{defn} \label{d2.4}
Let $k\subset L$ be a field extension. The field $L$ is rational
over $k$ (in short, $k$-rational) if, for some $n$, $L\simeq
k(X_1,\ldots,X_n)$, the rational function field of $n$ variables
over $k$. $L$ is called stably rational over $k$ (or, stably
$k$-rational) if the field $L(Y_1,\ldots,Y_m)$ is $k$-rational
where $Y_1,\ldots,Y_m$ are some elements algebraically independent
over $L$. When $k$ is an infinite field, $L$ is called retract
$k$-rational, if there exist an affine domain $A$ whose quotient
field is $L$ and $k$-algebra morphisms $\varphi: A\to
k[X_1,\ldots,X_n][1/f]$, $\psi: k[X_1,\ldots,X_n][1/f]\to A$
satisfying $\psi\circ\varphi=1_A$ the identity map on $A$ where
$k[X_1,\ldots,X_n]$ is a polynomial ring over $k$, $f\in
k[X_1,\ldots,X_n]\backslash \{0\}$ \cite[Definition 3.1; Ka,
Definition 1.1]{Sa}.
\end{defn}

It is known that ``$k$-rational" $\Rightarrow$ ``stably
$k$-rational" $\Rightarrow$ ``retract $k$-rational". Moreover, if
$k$ is an algebraic number field, retract $k$-rationality of
$k(G)$ implies the inverse Galois problem is true for the field
$k$ and the group $G$ \cite [Ka]{Sa} (see Definition \ref{d6.1}
for the definition of $k(G)$).

\begin{theorem} \label{t2.5}
Let $K/k$ be a finite Galois extension field, $G=\fn{Gal}(K/k)$
and $M$ be a $G$-lattice.
\begin{enumerate}
\item[{\rm (1)}] {\rm (Voskresenskii, Endo and Miyata
\cite[Theorem 1.2; Len, Theorem 1.7]{EM1})} $K(M)^G$ is stably
$k$-rational if and only if $[M]^{fl}$ is permutation, i.e.\ there
exists a short exact sequence of $G$-lattices $0\to M\to P_1\to
P_2\to 0$ where $P_1$ and $P_2$ are permutation $G$-lattices.
\item[{\rm (2)}] {\rm (Saltman \cite[Theorem 3.14; Ka, Theorem
2.8]{Sa})} Assume that $k$ is an infinite field. Then $K(M)^G$ is retract $k$-rational if and only if
$[M]^{fl}$ is invertible.
\end{enumerate}
\end{theorem}

\begin{theorem}[{Endo and Miyata \cite[Theorem 1.5; Sw2, Theorem 4.4]{EM2}}] \label{t2.6}
Let $G$ be a finite group. Then all the flabby (resp. coflabby)
$G$-lattices are invertible $\Leftrightarrow$ all the Sylow
subgroups of $G$ are cyclic $\Leftrightarrow$ $[I_G^0]^{fl}$ is
invertible where $I_G$ is the augmentation ideal of the group ring
$\bm{Z}[G]$.
\end{theorem}

\begin{theorem}[{Endo and Miyata \cite[Theorem 2.3]{EM2}}] \label{t2.7}
Let $G$ be a finite group, $I_G$ be the augmentation ideal of the
group ring $\bm{Z}[G]$. Then $[I_G^0]^{fl}=0 \Leftrightarrow [I_G^0]^{fl} \in F_G$ is of finite order $\Leftrightarrow G$ is isomorphic to the cyclic group $C_n$ where $n$ is any positive
integer, or the group $\langle \rho, \sigma, \tau:
\rho^m=\sigma^n=\tau^{2^d}=1, \tau \sigma \tau^{-1}=\sigma^{-1},
\rho \sigma=\sigma \rho, \rho \tau =\tau \rho \rangle$ where $m$
and $n$ are odd integers, $n\ge 3$, $d \ge 1$ and $\gcd
\{m,n\}=1$.
\end{theorem}

\section{\boldmath Some $D_n$-lattices}

Throughout this section,
$G$ denotes the group $G=\langle\sigma,\tau:\sigma^n=\tau^2=1,\tau\sigma\tau^{-1}=\sigma^{-1}\rangle$ where $n \ge 3$ is an odd integer,
i.e.\ $G$ is the dihedral group $D_n$.
Define $H=\langle \tau \rangle$.

We will construct six $G$-lattices which will become indecomposable $G$-lattices if $n=p$ is an odd prime number.

\begin{defn} \label{d3.1}
Let $G=\langle\sigma,\tau\rangle$ be the dihedral group defined
before. Define $G$-lattices $M_+$ and $M_-$ by $M_+=\fn{Ind}^G_H
\bm{Z}$, $M_-=\fn{Ind}^G_H \bm{Z}_-$ the induced lattices where
$\bm{Z}$ and $\bm{Z}_-$ are $H$-lattices such that $\tau$ acts on
$\bm{Z}= \bm{Z}\cdot u$, $\bm{Z}_-=\bm{Z}\cdot u'$ by $\tau\cdot
u=u$, $\tau\cdot u'=-u'$ respectively (note that $u$ and $u'$ are
the generators of $\bm{Z}$ and $\bm{Z}_-$ as abelian groups). By
choosing a $\bm{Z}$-basis for $M_+$ corresponding to $\sigma^i
u\in \fn{Ind}^G_H\bm{Z}$ (where $1\le i\le n$), the actions of
$\sigma$ and $\tau$ on $M_+$ are given by the $n\times n$ integral
matrices
\[
\sigma\mapsto A=\begin{pmatrix}
0 & 0 & \cdots & 0 & 1 \\
1 & 0 & & & 0 \\
& \ddots & \ddots & & \vdots \\
& & 1 & 0 & 0 \\
& & & 1 & 0
\end{pmatrix}, \quad
\tau\mapsto B=\left(\begin{array}{@{}cccc;{3pt/2pt}c@{}}
& & & 1 & \\ & & \iddots & & \\ & 1 & & & \\ 1 & & & & \\ \hdashline[3pt/2pt] & & & & 1
\end{array}\right).
\]
Note that the action of matrices is determined on column vectors.

Similarly, for a $\bm{Z}$-basis for $M_-$ corresponding to $\sigma^i u'$,
the actions of $\sigma$ and $\tau$ are given by
\[
\sigma\mapsto A=\begin{pmatrix}
0 & 0 & \cdots & 0 & 1 \\
1 & 0 & & & 0 \\
& \ddots & \ddots & & \vdots \\
& & 1 & 0 & 0 \\
& & & 1 & 0
\end{pmatrix}, \quad
\tau\mapsto -B=\left(\begin{array}{@{}cccc;{3pt/2pt}c@{}}
& & & -1 & \\ & & \iddots & & \\ & -1 & & & \\ -1 & & & & \\ \hdashline[3pt/2pt] & & & & -1
\end{array}\right).
\]
\end{defn}

\begin{defn} \label{d3.2}
As before, $G=\langle \sigma,\tau\rangle \simeq D_n$. Let
$f(\sigma)=1+\sigma+\cdots+\sigma^{n-1}\in \bm{Z}[G]$. Since
$\tau\cdot f(\sigma)=f(\sigma)\cdot \tau$, the left
$\bm{Z}[G]$-ideal $\bm{Z}[G]\cdot f(\sigma)$ is a two-sided ideal;
as an ideal in $\bm{Z}[G]$, we denote it by $\langle
f(\sigma)\rangle$. The natural projection $\bm{Z}[G]\to
\bm{Z}[G]/\langle f(\sigma)\rangle$ induces an isomorphism of
$\bm{Z}[G]/\langle f(\sigma)\rangle$ and the twisted group ring
$\bm{Z}[\zeta_n]\circ H$ (see \cite[p.589]{CR1}). Explicitly, let
$\zeta_n$ be a primitive $n$-th root of unity. Then
$\bm{Z}[\zeta_n]\circ H=\bm{Z}[\zeta_n]\oplus \bm{Z}[\zeta_n]\cdot
\tau$ and $\tau\cdot \zeta_n=\zeta_n^{-1}$. If $n=p$ is an odd
prime number, $\bm{Z}[\zeta_p]\circ H$ is a hereditary order
\cite[pages 593--595]{CR1}. Note that we have the following fibre
product diagram
\[
\xymatrix{\bm{Z}[G] \ar[r] \ar[d] & \bm{Z}[\zeta_n]\circ H \ar[d] \\
\bm{Z}[H] \ar[r] & \overline{\bm{Z}}[H]}
\]
where $\overline{\bm{Z}}=\bm{Z}/n\bm{Z}$ (compare with \cite[page 748, (34.43)]{CR1}).
\end{defn}

Using the $G$-lattices $M_+$ and $M_-$ in Definition \ref{d3.1},
define $N_+=\bm{Z}[G]/\langle f(\sigma)\rangle \otimes_{\bm{Z}[G]} M_+ =M_+/f(\sigma)M_+$,
$N_-=\bm{Z}[G]/\langle f(\sigma)\rangle \otimes_{\bm{Z}[G]} M_-=M_-/f(\sigma) M_-$.

The $\bm{Z}[G]/\langle f(\sigma)\rangle$-lattices $N_+$ and $N_-$ may be regarded as $G$-lattices
through the $\bm{Z}$-algebra morphism $\bm{Z}[G]\to \bm{Z}[G]/\langle f(\sigma)\rangle$.
By choosing a $\bm{Z}$-basis for $N_+$ corresponding to $\sigma^i u$ where $1\le i\le n-1$,
the actions of $\sigma$ and $\tau$ on $N_+$ are given by the $(n-1)\times (n-1)$ integral matrices
\[
\sigma\mapsto A'=\begin{pmatrix}
0 & 0 & 0 & \cdots & 0 & -1 \\
1 & 0 & & & & -1 \\
& 1 & 0 & & & -1 \\
& & & \ddots & & \vdots \\
& & & & 0 & -1 \\
& & & & 1 & -1
\end{pmatrix}, \quad
\tau\mapsto B'=\begin{pmatrix}
& & & 1 \\ & & \iddots & \\ & 1 & & \\ 1 & & &
\end{pmatrix}.
\]

Similarly, the actions of $\sigma$ and $\tau$ on $N_-$ are given by
\[
\sigma\mapsto A'=\begin{pmatrix}
0 & 0 & 0 & \cdots & 0 & -1 \\
1 & 0 & & & & -1 \\
& 1 & 0 & & & -1 \\
& & & \ddots & & \vdots \\
& & & & 0 & -1 \\
& & & & 1 & -1
\end{pmatrix}, \quad
\tau\mapsto -B'=\begin{pmatrix}
& & & -1 \\ & & \iddots & \\ & -1 & & \\ -1 & & &
\end{pmatrix}.
\]

\begin{defn} \label{d3.3}
We will use the $G$-lattices $M_+$ and $M_-$ in Definition
\ref{d3.1} to construct $G$-lattices $\widetilde{M}_+$ and
$\widetilde{M}_-$ which are of rank $n+1$ satisfying the short
exact sequences of $G$-lattices
\begin{align*}
0 &\to M_+\to \widetilde{M}_+\to \bm{Z}_-\to 0 \\
0 &\to M_-\to \widetilde{M}_-\to \bm{Z}\to 0
\end{align*}
where the $\bm{Z}$-lattice structures of $\widetilde{M}_+$ and $\widetilde{M}_-$ will be described below and $\bm{Z}=\bm{Z}\cdot w$,
$\bm{Z}_-=\bm{Z}\cdot w'$ are $G$-lattices defined by $\sigma\cdot w=w$, $\tau\cdot w=w$, $\sigma\cdot w'=w'$, $\tau\cdot w'=-w'$.
\end{defn}

Let $\{w_i=\sigma^{i+1}u:0\le i\le n-1\}$ be the corresponding $\bm{Z}$-basis of $M_+$ in Definition \ref{d3.1}.
As a free abelian group, $\widetilde{M}_+=(\bigoplus_{0\le i\le n-1} \bm{Z}\cdot w_i) \oplus \bm{Z}\cdot w$.
Define the actions of $\sigma$ and $\tau$ on $\widetilde{M}_+$ by the $(n+1)\times (n+1)$ integral matrices
\[
\sigma\mapsto \tilde{A}=\left(\begin{array}{@{\quad}c@{\quad~};{3pt/2pt}c@{}}
 & \\ A & \\  & \\ \hdashline[3pt/2pt] & 1
\end{array}\right),\quad
\tau\mapsto \tilde{B}=\left(\begin{array}{@{\quad}c@{\quad~};{3pt/2pt}c@{}}
 & 1 \\ B & \vdots \\ & 1 \\ \hdashline[3pt/2pt] & -1
\end{array}\right).
\]

Similarly, let $\{w_i^{\prime}=\sigma^{i+1}u^{\prime}:0\le i\le n-1\}$ be the $\bm{Z}$-basis of $M_-$ in Definition \ref{d3.1},
and $\widetilde{M}_-=(\bigoplus_{0\le i\le n-1} \bm{Z} w_i^{\prime})\oplus\bm{Z}\cdot w'$.
Define the actions of $\sigma$ and $\tau$ on $\widetilde{M}_-$ by
\[
\sigma\mapsto \tilde{A}=\left(\begin{array}{@{\quad}c@{\quad~};{3pt/2pt}c@{}}
 & \\ A & \\  & \\ \hdashline[3pt/2pt] & 1
\end{array}\right),\quad
\tau\mapsto -\tilde{B}=\left(\begin{array}{@{\quad}c@{\quad~};{3pt/2pt}c@{}}
 & -1 \\ -B & \vdots \\ & -1 \\ \hdashline[3pt/2pt] & 1
\end{array}\right).
\]

In this section, we will show that
$\widetilde{M}_+$ and $\widetilde{M}_-$ are stably permutation
$G$-lattices.

\begin{theorem} \label{t3.4}
Let $G=\langle\sigma,\tau:\sigma^n=\tau^2=1,\tau\sigma\tau^{-1}=\sigma^{-1}\rangle \simeq D_n$ where $n$ is an odd integer.
Then $\widetilde{M}_+\oplus \bm{Z}\simeq \bm{Z}[G/\langle\sigma\rangle]\oplus\bm{Z}[G/\langle\tau\rangle]$.
\end{theorem}

\begin{proof}
Let $u_0$, $u_1$ be the $\bm{Z}$-basis of
$\bm{Z}[G/\langle\sigma\rangle]$ corresponding to the cosets $\langle \sigma \rangle$, $\tau\langle \sigma \rangle$. Then
$\sigma:u_0\mapsto u_0$, $u_1\mapsto u_1$,
$\tau:u_0\leftrightarrow u_1$.

Let $\{v_i:0\le i\le n-1\}$ be the $\bm{Z}$-basis of
$\bm{Z}[G/\langle\tau\rangle]$ correspond to $\sigma^i\langle\tau\rangle$ where
$0\le i\le n-1$. Then $\sigma: v_i\mapsto v_{i+1}$ (where the
index is understood modulo $n$), $\tau: v_i\mapsto v_{n-i}$ for
$0\le i\le n-1$.

It follows that $u_0,u_1,v_0,v_1,\ldots,v_{n-1}$ is a $\bm{Z}$-basis of $\bm{Z}[G/\langle\sigma\rangle]\oplus\bm{Z}[G/\langle\tau\rangle]$.

Define
\begin{gather*}
t = u_0+u_1+\sum_{0\le i\le n-1} v_i, \\
x = u_0+u_1+\sum_{1\le i\le n-1} v_i,\quad y= \frac{n-1}{2}u_0+\frac{n+1}{2}u_1+\frac{n-1}{2} \sum_{0\le i\le n-1}v_i.
\end{gather*}
Note that $x=t-v_0$ and $y=\frac{n-1}{2}t+u_1$.

Since $\sum_{0\le i\le n-1} \sigma^i\cdot (v_1+\cdots+v_{n-1})=(n-1)\sum_{0\le i\le n-1} v_i$,
it follows that $\tau\cdot y=-y+\sum_{0\le i\le n-1} \sigma^i (x)$. It is routine to verify that
\[
\left(\bigoplus_{0\le i\le n-1} \bm{Z}\cdot \sigma^i(x)\right) \oplus \bm{Z}\cdot y \simeq \widetilde{M}_+, \quad \bm{Z}\cdot t\simeq \bm{Z}
\]
by checking the actions of $\sigma$ and $\tau$ on lattices in both sides.

Now we will show that $\sigma(x),
\sigma^2(x),\ldots,\sigma^{n-1}(x),x,y,t$ is a $\bm{Z}$-basis of
$\bm{Z}[G/\langle\sigma\rangle]\oplus\bm{Z}[G/\langle\tau\rangle]$.
Write the determinant of these $n+2$ elements with respect to the
$\bm{Z}$-basis $u_0,u_1,v_0,v_1,\ldots,v_{n-1}$. We get the
coefficient matrix $T$ as
\begin{gather*}
\renewcommand{\arraystretch}{1.15} \normalsize
T=\left(\begin{array}{@{\,}cccc;{3pt/2pt}cc@{\,}}
1 & 1 & \cdots & 1 & \frac{n-1}{2} & 1 \\
1 & 1 & \cdots & 1 & \frac{n+1}{2} & 1 \\[1mm] \hdashline[3pt/2pt]
1 & 1 & & 0 & \frac{n-1}{2} & 1 \\
0 & 1 & & 1 & \frac{n-1}{2} & 1 \\
1 & 0 & & 1 & \frac{n-1}{2} & 1 \\
\vdots & \vdots & & \vdots & \vdots & \vdots \\
1 & 1 & & 1 & \frac{n-1}{2} & 1
\end{array}\right) \hspace*{-44mm}
\begin{minipage}[c][0mm]{56mm}
$\begin{blockarray}{cl}
~ & ~ \\[19mm]
\begin{block}{c\}l}
~ & \\ & \\ & \text{$n$ rows} \\[2pt] & \\ & \\
\end{block}
~ & ~ \\[-10mm] \underbrace{\hspace*{20mm}}_{\text{\normalsize$n$ columns}}\hspace*{21mm} &
\end{blockarray}$
\end{minipage} \\[1mm]
\end{gather*}

The determinant of $T$ may be calculated as follows: Subtract the
last column from each of the first $n$ columns. Also subtract
$\frac{n-1}{2}$ times of the last column from the $(n+1)$-th
column. Then it is easy to see $\det (T)= 1$.
\end{proof}

\begin{theorem} \label{t3.5}
Let $G=\langle\sigma,\tau:\sigma^n=\tau^2=1,\tau\sigma\tau^{-1}=\sigma^{-1}\rangle\simeq D_n$ where $n$ is an odd integer.
Then $\widetilde{M}_-\oplus\bm{Z}[G/\langle\tau\rangle]\simeq \bm{Z}[G]\oplus\bm{Z}$.
\end{theorem}

\begin{proof}
The idea of the proof is similar to that of Theorem \ref{t3.4}.

Let $u_0,u_1,\ldots,u_{n-1},v_0,v_1,\ldots,v_{n-1},t$ be a $\bm{Z}$-basis of $\bm{Z}[G]\oplus\bm{Z}$ where $u_i$,
$v_j$ correspond to $\sigma^i$, $\sigma^j\tau$ in $\bm{Z}[G]$.
The actions of $\sigma$ and $\tau$ are given by
\begin{align*}
\sigma &: u_i\mapsto u_{i+1},~ v_j\mapsto v_{j+1},~ t\mapsto t, \\
\tau &: u_i\leftrightarrow v_{n-i},~ t\mapsto t
\end{align*}
where the index of $u_i$ or $v_j$ is understood modulo $n$.

Define $x,y,z\in \bm{Z}[G]\oplus \bm{Z}$ by
\begin{gather*}
x = u_0-v_0, \quad y=\left(\sum_{0\le i\le n-1} u_i\right)+t, \\
z =\left(\sum_{1\le i\le \frac{n-1}{2}} u_i\right)+\left(\sum_{\frac{n+1}{2}\le j\le n-1} v_j\right)+t.
\end{gather*}

We claim that
\[
\left(\bigoplus_{0\le i\le n-1} \bm{Z}\cdot \sigma^i(x)\right)\oplus \bm{Z}\cdot y\simeq \widetilde{M}_-, \quad
\bigoplus_{0\le i\le n-1} \bm{Z}\cdot \sigma^i(z)\simeq \bm{Z}[G/\langle\tau\rangle].
\]

Since $\tau(x)=-x$, $\tau(z)=z$,
it follows that $\tau\cdot \sigma^i(x)=-\sigma^{n-i}(x)$, $\tau\cdot\sigma^i(z)=\sigma^{n-i}(z)$.
The remaining proof is omitted.

Now we will show that
$\sigma(x),\sigma^2(x),\ldots,\sigma^{n-1}(x),
x,y,\sigma(z),\ldots,\sigma^{n-1}(z)$, $z$ form a $\bm{Z}$-basis
of $\bm{Z}[G]\oplus \bm{Z}$. Write the coefficient matrix of these
elements with respect to the $\bm{Z}$-basis
$u_0,u_1,\ldots,u_{n-1},v_0,v_1,\ldots,v_{n-1},t$. We get
$\det(T_n)$ where $T_n$ is a $(2n+1)\times (2n+1)$ integral
matrix. For example,
\[
T_3=\left(\begin{array}{@{}ccc;{3pt/2pt}c;{3pt/2pt}ccc@{}}
0 & 0 & 1 & 1 & 0 & 0 & 1 \\
1 & 0 & 0 & 1 & 1 & 0 & 0 \\
0 & 1 & 0 & 1 & 0 & 1 & 0 \\ \hdashline[3pt/2pt]
0 & 0 & -1 & 0 & 0 & 1 & 0 \\
-1 & 0 & 0 & 0 & 0 & 0 & 1 \\
0 & -1 & 0 & 0 & 1 & 0 & 0 \\ \hdashline[3pt/2pt]
0 & 0 & 0 & 1 & 1 & 1 & 1
\end{array}\right).
\]

For any $n$, we evaluate $\det (T_n)$ by adding the $i$-th row to
$(i+n)$-th row of $T_n$ for $1\le i\le n$. We find that
$\det(T_n)=\pm\det(T')$ where $T'$ is an $(n+1)\times (n+1)$
integral matrix. Note that all the entries of the $i$-th row of
$T'$ (where $1\le i\le n$) are one except one position, because of
the definition of $z$ (and those of $\sigma^i(z)$ for $1\le i\le
n-1$). Subtract the last row from the $i$-th row where $1\le i\le
n$. We find $\det(T')=\pm 1$.
\end{proof}

Suppose that $G\simeq D_n$ where $n$ is an odd integer. From Theorem \ref{t3.4} and Theorem \ref{t3.5}, we find that
\[
\widetilde{M}_+ \oplus \widetilde{M}_-
\oplus \bm{Z} \oplus \bm{Z}[G/\langle\tau\rangle]\simeq
\bm{Z}[G] \oplus \bm{Z}[G/\langle\sigma\rangle]\oplus
\bm{Z} \oplus \bm{Z}[G/\langle\tau\rangle].
\]
If the cancellation of $\bm{Z} \oplus \bm{Z}[G/\langle \tau \rangle]$ is possible, then we get the isomorphism
\[
\widetilde{M}_+\oplus \widetilde{M}_-\simeq \bm{Z}[G]\oplus \bm{Z}[G/\langle \sigma \rangle].
\]
In Theorem \ref{t3.7}, we will show that this is, indeed, the case.

We first prove Lemma \ref{l3.6}.
Define
\begin{align*}
{\rm Circ}(c_0,c_1,\ldots,c_{n-1})=
\left(
\begin{array}{@{}ccccc@{}}
 c_0 & c_{n-1} & \cdots & c_2 & c_1 \\
 c_1 & c_0 & c_{n-1} &  & c_2 \\
 \vdots & c_1 & c_0 & \ddots & \vdots \\
 c_{n-2} &  & \ddots & \ddots & c_{n-1} \\
 c_{n-1} & c_{n-2} & \cdots & c_1 & c_0
\end{array}
\right)
\end{align*}
be the $n\times n$ circulant matrix whose determinant is
\begin{align*}
\det({\rm Circ}(c_0,c_1,\ldots,c_{n-1}))
&=\prod_{k=0}^{n-1}(c_0+c_1\zeta_n^k+\cdots+c_{n-1}\zeta_n^{(n-1)k}).
\end{align*}
The determinant of the circulant matrix may be found, for examples, on \cite[page 281]{Ja}.

\begin{lemma} \label{l3.6}
Let $n\geq 3$ be an odd integer.\\
{\rm (1)}
$\displaystyle{\det({\rm Circ}(\overbrace{1,\ldots,1}^\frac{n-1}{2},
\overbrace{0,\ldots,0}^\frac{n+1}{2}))=\frac{n-1}{2}}$. \\
{\rm (2)}
$\det({\rm Circ}(\overbrace{-1,\ldots,-1}^\frac{n-1}{2},0,\overbrace{1,\ldots,1}^\frac{n-3}{2},0))=-1$.
\end{lemma}

\begin{proof}
(1)
\begin{align*}
&\det({\rm Circ}(\overbrace{1,\ldots,1}^\frac{n-1}{2},
\overbrace{0,\ldots,0}^\frac{n+1}{2}))\\
&=\prod_{k=0}^{n-1}\left(1+\zeta_n^k+\cdots+\zeta_n^{\frac{n-3}{2}k}\right)\\
&=\frac{n-1}{2} \prod_{k=1}^{n-1}\left(1+\zeta_n^k+\cdots+\zeta_n^{\frac{n-3}{2}k}\right)\\
&=\frac{n-1}{2} \prod_{k=1}^{n-1}\frac{(1-(\zeta_n^k)^{\frac{n-1}{2}})}{(1-\zeta_n^k)}=\frac{n-1}{2}\\
\end{align*}
where in the $3$rd line, we apply $\sum_{i=0}^mx^i=\frac{1-x^{m+1}}{1-x}$ 
for $x=\zeta_n^k$ and $m=\frac{n-3}{2}$ and 
in the $4$th line, we note that $\frac{n-1}{2}$ and $n$ 
are relatively prime, which implies that 
$\langle\zeta_n^{\frac{n-1}{2}}\rangle=\langle\zeta_n\rangle$ and so 
\begin{align*}
\prod_{k=1}^{n-1}(1-(\zeta_n^{\frac{n-1}{2}})^k)=\prod_{k=1}^{n-1}(1-\zeta_n^k)
\end{align*} 
as required.\\ 

(2) By the same argument as in (1), we have
\begin{align*}
&\det({\rm Circ}(\overbrace{-1,\ldots,-1}^\frac{n-1}{2},0,\overbrace{1,\ldots,1}^\frac{n-3}{2},0))\\
&=\prod_{k=0}^{n-1}\left(-1-\zeta_n^k-\cdots-\zeta_n^{\frac{n-3}{2}k}+\zeta_n^{\frac{n+1}{2}k}+\cdots+\zeta_n^{(n-2)k}\right)\\
&=\left(-\frac{n-1}{2}+\frac{n-3}{2}\right)(-1)^{n-1}\prod_{k=1}^{n-1}\left(1-\zeta_n^{\frac{n+1}{2}k}+\zeta_n^k-\zeta_n^{\frac{n+3}{2}k}
+\cdots-\zeta_n^{(n-2)k}+\zeta_n^{\frac{n-3}{2}k}\right)\\
&=(-1)^n\prod_{k=1}^{n-1}\sum_{i=0}^{n-3}(-\zeta_n^{\frac{n+1}{2}k})^i\\
&=(-1)\prod_{k=1}^{n-1}\frac{(1-(-\zeta_n^{\frac{n+1}{2}k})^{n-2})}{(1-(-\zeta_n^{\frac{n+1}{2}k}))}\\
&=(-1)\prod_{k=1}^{n-1}\frac{(1+\zeta_n^{\frac{n+1}{2}(n-2)k})}{(1+\zeta_n^{\frac{n+1}{2}k})}=-1
\end{align*}
where we note that $\frac{n+1}{2}(2s)\equiv s\pmod{n}$ and 
$\frac{n+1}{2}(2s+1)\equiv(\frac{n+1}{2}+s)\pmod{n}$ for the 3rd line and 
$n-2$ is odd for the 5th line. 
We again note that $\frac{n+1}{2}$ is relatively prime to $n$ 
and also $\frac{n+1}{2}(n-2)$ is relatively prime to $n$ so 
$\langle\zeta_n^{\frac{n+1}{2}(n-2)}\rangle=\langle\zeta_n^{\frac{n+1}{2}}\rangle=\langle\zeta_n\rangle$ allows us to conclude as before that the last product is $1$.
\end{proof}

\begin{theorem} \label{t3.7}
Let $G=\langle\sigma,\tau: \sigma^n=\tau^2=1,\tau\sigma\tau^{-1}=\sigma^{-1}\rangle\simeq D_n$ where $n$ is an odd integer.
Then $\widetilde{M}_+\oplus \widetilde{M}_-\simeq \bm{Z}[G]\oplus \bm{Z}[G/\langle \sigma \rangle]$.
\end{theorem}

\begin{proof}
Let $u_0,u_1,\ldots,u_{n-1},v_0,v_1,\ldots,v_{n-1}$ be
a $\bm{Z}$-basis of $\bm{Z}[G]$ where $u_i$,
$v_j$ correspond to $\sigma^i$, $\sigma^j\tau$ in $\bm{Z}[G]$.
The actions of $\sigma$ and $\tau$ are given by
\begin{align*}
\sigma &: u_i\mapsto u_{i+1},~ v_j\mapsto v_{j+1},\\
\tau &: u_i\leftrightarrow v_{n-i}
\end{align*}
where the index of $u_i$ or $v_j$ is understood modulo $n$. Let
$t_0$, $t_1$ be the $\bm{Z}$-basis of
$\bm{Z}[G]/\langle\sigma\rangle]$ correspond to the cosets $\langle \sigma \rangle$, $\tau\langle \sigma \rangle$. Then
$\sigma:t_0\mapsto t_0$, $t_1\mapsto t_1$,
$\tau:t_0\leftrightarrow t_1$.

The idea first is to find isomorphic copies of $\widetilde{M}_+$ 
and $\widetilde{M}_-$ inside $\bm{Z}[G]\oplus\bm{Z}[G/\langle\sigma\rangle]$ 
then to show that the $\bm{Z}$-basis of $\widetilde{M}_+\oplus\widetilde{M}_-$ 
is a $\bm{Z}$-basis of $\bm{Z}[G]\oplus\bm{Z}[G/\langle\sigma\rangle]$. 

To find an isomorphic copy of $\widetilde{M}_+$, 
one could find an element $x_0$ with $\tau(x_0)=x_0$ 
such that 
$\{\sigma^i x_0 : i=0,\ldots,n-1\}$ is $\bm{Z}$-linearly independent. 
Then it can easily be shown that $\bm{Z}[D_n]\cdot x_0\simeq M_+$. 
If we find $y_0$ such that $\sigma(y_0)=y_0$ and 
$\tau(y_0)=-y_0+\sum_{i=0}^{n-1}\sigma^i(x_0)$, then 
$\{\sigma^i x_0 : i=0,\ldots,n-1\}\cup\{y_0\}$ is a $\bm{Z}$-basis 
of $\widetilde{M}_+$. 
Indeed, we may take such $x_0$ and $y_0$ as 
\begin{align*}
x_0&=\sum_{i=\frac{n+1}{2}}^{n-1}u_i
+\sum_{j=1}^{\frac{n-1}{2}}v_j+t_0+t_1,\\
y_0&=\frac{n-1}{2}\left(\sum_{j=0}^{n-1} v_j\right)+t_0+(n-1)t_1.
\end{align*}

Similary, to find an isomorphic copy of $\widetilde{M}_-$, 
one could find an element $z_0$ with $\tau(z_0)=-z_0$ 
such that 
$\{\sigma^i z_0 : i=0,\ldots,n-1\}$ is $\bm{Z}$-linearly independent. 
Then it can easily be shown that $\bm{Z}[D_n]\cdot z_0\simeq M_-$. 
If we find $y_1$ such that $\sigma(y_1)=y_1$ and 
$\tau(y_1)=y_1-\sum_{i=0}^{n-1}\sigma^i(z_0)$, then 
$\{\sigma^i z_0 : i=0,\ldots,n-1\}\cup\{y_1\}$ is a $\bm{Z}$-basis 
of $\widetilde{M}_-$. 
Indeed, we may take such $z_0$ and $y_1$ as 
\begin{align*}
z_0 &=u_0+\sum_{i=\frac{n+1}{2}}^{n-1} u_i
-\sum_{j=0}^{\frac{n-1}{2}} v_j+t_0-t_1,\\
y_1&=\sum_{i=0}^{n-1}u_i
-\frac{n-1}{2}\left(\sum_{j=0}^{n-1} v_j\right)+t_0-(n-1)t_1.
\end{align*}
Then we have 
\[
\left(\bigoplus_{0\le i\le n-1} \bm{Z}\cdot \sigma^i(x_0)\right)\oplus \bm{Z}\cdot y_0\simeq \widetilde{M}_+, \quad
\left(\bigoplus_{0\le i\le n-1} \bm{Z}\cdot \sigma^i(z_0)\right)\oplus \bm{Z}\cdot y_1\simeq \widetilde{M}_-.
\]

It remains to show that $\sigma^{\frac{n-1}{2}}(x_0)$, $\ldots$,
$\sigma^{n-1}(x_0)$, $x_0$, $\sigma(x_0)$, $\ldots$,
$\sigma^{\frac{n-3}{2}}(x_0)$, $y_0$, $\sigma^{\frac{n-1}{2}}(z_0)$,
$\ldots$, $\sigma^{n-1}(z_0)$, $z_0$, $\sigma(z_0)$, $\ldots$,
$\sigma^{\frac{n-3}{2}}(z_0)$, $y_1$ form a $\bm{Z}$-basis of
$\bm{Z}[G]\oplus \bm{Z}[G/\langle\sigma\rangle]$. 

Let $Q$ be the 
coefficient matrix of $\sigma^{\frac{n-1}{2}}(x_0)$, $\ldots$,
$\sigma^{n-1}(x_0)$, $x_0$, $\sigma(x_0)$, $\ldots$,
$\sigma^{\frac{n-3}{2}}(x_0)$, $y_0$, $\sigma^{\frac{n-1}{2}}(z_0)$,
$\ldots$, $\sigma^{n-1}(z_0)$, $z_0$, $\sigma(z_0)$, $\ldots$,
$\sigma^{\frac{n-3}{2}}(z_0)$, $y_1$ with respect to the
$\bm{Z}$-basis $u_0, u_1,\ldots,u_{n-1},v_0, v_1,\ldots,v_{n-1},
t_0,t_1$.

The matrix $Q$ is defined as
\begin{align*}
Q=
\left(\begin{array}{@{}ccc;{3pt/2pt}c;{3pt/2pt}ccc;{3pt/2pt}c@{}}
 & & & 0 & & & & 1\\[-3mm]
\multicolumn{3}{c;{3pt/2pt}}{{\rm Circ}(\overbrace{1,\ldots,1}^\frac{n-1}{2},
\overbrace{0,\ldots,0}^\frac{n+1}{2})} & \vdots &
\multicolumn{3}{c;{3pt/2pt}}{{\rm Circ}(\overbrace{1,\ldots,1}^\frac{n+1}{2},
\overbrace{0,\ldots,0}^\frac{n-1}{2})} & \vdots\\[1mm]
 & & & 0 & & & & 1\\ \hdashline[3pt/2pt] & & & & & & & \\[-4mm]
 & & & \frac{n-1}{2} & & & & -\frac{n-1}{2}\\[-3mm]
\multicolumn{3}{c;{3pt/2pt}}{{\rm Circ}(\overbrace{0,\ldots,0}^\frac{n+1}{2},
\overbrace{1,\ldots,1}^\frac{n-1}{2})} & \vdots &
\multicolumn{3}{c;{3pt/2pt}}{{\rm Circ}(\overbrace{0,\ldots,0}^\frac{n-1}{2},
\overbrace{-1,\ldots,-1}^\frac{n+1}{2})} & \vdots\\[1mm]
 & & & \frac{n-1}{2} & & & & -\frac{n-1}{2}\\[1mm] \hdashline[3pt/2pt]
\multicolumn{3}{c;{3pt/2pt}}{1\qquad\quad \cdots\qquad\quad 1} & 1 & \multicolumn{3}{c;{3pt/2pt}}{1\qquad\quad \cdots\qquad\quad 1} & 1\\ \hdashline[3pt/2pt]
\multicolumn{3}{c;{3pt/2pt}}{\underbrace{1\qquad\quad \cdots\qquad\quad 1}_{n}} & n-1 & \multicolumn{3}{c;{3pt/2pt}}{\underbrace{-1\qquad\ \ \cdots\qquad\ \ -1}_{n}}  & -(n-1)
\end{array}\right)
\hspace*{-3mm}\begin{blockarray}{cl}
\begin{block}{c\}l}
 & \\[2mm] & n \\[2mm] & \\
\end{block} \\[-4mm]
\begin{block}{c\}l}
 & \\[2mm] & n \\[2mm] & \\
\end{block}
 &  \\ & \\
\end{blockarray}.
\end{align*}

For examples, when $n=3,5$, $Q$ is of the form
\[
\left(
\begin{array}{@{}ccc;{3pt/2pt}c;{3pt/2pt}ccc;{3pt/2pt}c@{}}
 1 & 0 & 0 & 0 & 1 & 0 & 1 & 1 \\
 0 & 1 & 0 & 0 & 1 & 1 & 0 & 1 \\
 0 & 0 & 1 & 0 & 0 & 1 & 1 & 1 \\\hdashline[3pt/2pt]
 0 & 1 & 0 & 1 & 0 & -1 & -1 & -1 \\
 0 & 0 & 1 & 1 & -1 & 0 & -1 & -1 \\
 1 & 0 & 0 & 1 & -1 & -1 & 0 & -1 \\\hdashline[3pt/2pt]
 1 & 1 & 1 & 1 & 1 & 1 & 1 & 1 \\\hdashline[3pt/2pt]
 1 & 1 & 1 & 2 & -1 & -1 & -1 & -2 \\
\end{array}
\right)
\]
and
\[
\left(
\begin{array}{@{}ccccc;{3pt/2pt}c;{3pt/2pt}ccccc;{3pt/2pt}c@{}}
 1 & 0 & 0 & 0 & 1 & 0 & 1 & 0 & 0 & 1 & 1 & 1 \\
 1 & 1 & 0 & 0 & 0 & 0 & 1 & 1 & 0 & 0 & 1 & 1 \\
 0 & 1 & 1 & 0 & 0 & 0 & 1 & 1 & 1 & 0 & 0 & 1 \\
 0 & 0 & 1 & 1 & 0 & 0 & 0 & 1 & 1 & 1 & 0 & 1 \\
 0 & 0 & 0 & 1 & 1 & 0 & 0 & 0 & 1 & 1 & 1 & 1 \\\hdashline[3pt/2pt]
 0 & 1 & 1 & 0 & 0 & 2 & 0 & -1 & -1 & -1 & 0 & -2 \\
 0 & 0 & 1 & 1 & 0 & 2 & 0 & 0 & -1 & -1 & -1 & -2 \\
 0 & 0 & 0 & 1 & 1 & 2 & -1 & 0 & 0 & -1 & -1 & -2 \\
 1 & 0 & 0 & 0 & 1 & 2 & -1 & -1 & 0 & 0 & -1 & -2 \\
 1 & 1 & 0 & 0 & 0 & 2 & -1 & -1 & -1 & 0 & 0 & -2 \\\hdashline[3pt/2pt]
 1 & 1 & 1 & 1 & 1 & 1 & 1 & 1 & 1 & 1 & 1 & 1 \\\hdashline[3pt/2pt]
 1 & 1 & 1 & 1 & 1 & 4 & -1 & -1 & -1 & -1 & -1 & -4 \\
\end{array}
\right).
\]

\ We will show that det$(Q)=-1$. For a given matrix, we denote by
$C(i)$ its $i$-th column. When we say that, apply $C(i)+C(1)$ on
the $i$-th column, we mean the column operation by adding the 1st
column to the $i$-th column.

On the $(2n+2)$-th column, apply $C(2n+2)+C(n+1)$. On the
$(n+1)$-th column, apply $C(n+1)+\frac{n-1}{2}(C(2n+2))$. On the
$(n+1)$-th column, apply $C(n+1)-C(1)-\cdots-C(n)$. Then all the
entries of the $(n+1)$-th column are zero except for the last
$(2n+2)$-th entry, which is $-1$. Hence it is enough to show
$\det(Q_0)=1$ where $Q_0$ is a $(2n+1)\times(2n+1)$ matrix defined
by
\begin{align*}
Q_0=\left(\begin{array}{@{}ccc;{3pt/2pt}ccc;{3pt/2pt}c@{}}
 & &  & & & & 1\\[-3mm]
\multicolumn{3}{c;{3pt/2pt}}{{\rm Circ}(\overbrace{1,\ldots,1}^\frac{n-1}{2},
\overbrace{0,\ldots,0}^\frac{n+1}{2})} &
\multicolumn{3}{c;{3pt/2pt}}{{\rm Circ}(\overbrace{1,\ldots,1}^\frac{n+1}{2},
\overbrace{0,\ldots,0}^\frac{n-1}{2})} & \vdots\\[1mm]
 & & & & & & 1\\\hdashline[3pt/2pt]
 & & & & & & 0\\[-3mm]
\multicolumn{3}{c;{3pt/2pt}}{{\rm Circ}(\overbrace{0,\ldots,0}^\frac{n+1}{2},
\overbrace{1,\ldots,1}^\frac{n-1}{2})} &
\multicolumn{3}{c;{3pt/2pt}}{{\rm Circ}(\overbrace{0,\ldots,0}^\frac{n-1}{2},
\overbrace{-1,\ldots,-1}^\frac{n+1}{2})} & \vdots\\[1mm]
 & & & & & & 0 \\\hdashline[3pt/2pt]
\multicolumn{3}{c;{3pt/2pt}}{\underbrace{1\qquad\quad\ \cdots\qquad\quad\ 1}_{n}}
 & \multicolumn{3}{c;{3pt/2pt}}{\underbrace{1\qquad\quad\ \cdots\qquad\quad\ 1}_{n}}  & 2
\end{array}\right)
\hspace*{-3mm}\begin{blockarray}{cl}
\begin{block}{c\}l}
 & \\[2mm] & n \\[2mm] & \\
\end{block}
\begin{block}{c\}l}
 & \\[2mm] & n \\[2mm] & \\
\end{block}
 &  \\
\end{blockarray}.
\end{align*}

On the $(n+i)$-th column, apply $C(n+i)+C(f(\frac{n+1}{2}+i))$ for
$i=1,\ldots n$ where
\[
f(k)=
\begin{cases}
k & k \leq n \\
k-n & k>n.
\end{cases}
\]
On the $(n+i)$-th column, apply $C(n+i)-C(2n+1)$ for
$i=1,\ldots,n$. On the $(2n+1)$-th column, apply
$C(2n+1)-\frac{2}{n-1}\{C(1)+\cdots+C(n)\}-\{C(n+1)+\cdots+C(2n)\}$. 
Thus we get $\det(Q_0)=\det(Q_1)$ where
\begin{align*}
Q_1=\left(\begin{array}{@{}ccc;{3pt/2pt}ccc;{3pt/2pt}c@{}}
 & &  & & & & 0\\[-3mm]
\multicolumn{3}{c;{3pt/2pt}}{{\rm Circ}(\overbrace{1,\ldots,1}^\frac{n-1}{2},
\overbrace{0,\ldots,0}^\frac{n+1}{2})} &
&  & \mathbf{O} &\vdots\\[1mm]
 & & & & & & 0\\\hdashline[3pt/2pt]
 & & & & & & 0\\[-3mm]
\multicolumn{3}{c;{3pt/2pt}}{{\rm Circ}(\overbrace{0,\ldots,0}^\frac{n+1}{2},
\overbrace{1,\ldots,1}^\frac{n-1}{2})} &
\multicolumn{3}{c;{3pt/2pt}}{{\rm Circ}(0,\overbrace{1,\ldots,1}^\frac{n-3}{2},0,
\overbrace{-1,\ldots,-1}^\frac{n-1}{2})} & \vdots\\[1mm]
 & & & & & & 0 \\\hdashline[3pt/2pt]
\multicolumn{3}{c;{3pt/2pt}}{\underbrace{1\qquad\quad \cdots\qquad\quad 1}_{n}}
 & \multicolumn{3}{c;{3pt/2pt}}{\underbrace{0\qquad\qquad \cdots\qquad\qquad 0}_{n}}  & -\frac{2}{n-1}
\end{array}\right)
\hspace*{-3mm}\begin{blockarray}{cl}
\begin{block}{c\}l}
 & \\[2mm] & n \\[2mm] & \\
\end{block}
\begin{block}{c\}l}
 & \\[2mm] & n \\[2mm] & \\
\end{block}
 &  \\
\end{blockarray}.
\end{align*}
Because of Lemma \ref{l3.6} and
\[
\det({\rm Circ}(0,\overbrace{1,\ldots,1}^\frac{n-3}{2},0,
\overbrace{-1,\ldots,-1}^\frac{n-1}{2})) =\det({\rm
Circ}(\overbrace{-1,\ldots,-1}^\frac{n-1}{2},0,\overbrace{1,\ldots,1}^\frac{n-3}{2},0)),
\] we find $\det(Q_1)=\frac{n-1}{2}\cdot(-1)\cdot(-\frac{2}{n-1})=1$.
\end{proof}

\begin{lemma} \label{l4.4}
Let $M_+$, $M_-$, $N_+$, $N_-$, $\bm{Z}$, $\bm{Z}_-$ be $G$-lattices with
$G=\langle\sigma,\tau:\sigma^n=\tau^2=1,
\tau\sigma\tau^{-1}=\sigma^{-1}\rangle\simeq D_n$ in Definition \ref{d3.1}, Definition \ref{d3.2} and  Definition \ref{d3.3} where $n$ is an
odd integer. Then there are non-split exact sequences of
$G$-lattices $0\to N_-\to M_+\to \bm{Z}\to 0$ and $0\to N_+\to M_-\to
\bm{Z}_-\to 0$.
\end{lemma}

\begin{proof}
\begin{Case}{1} $M_+$. \end{Case}

By definition, choose a $\bm{Z}$-basis $\{x_i:0\le i\le n-1\}$ of
$M_+$ such that $\sigma: x_i\mapsto x_{i+1}$,
$\tau:x_i\leftrightarrow x_{n-i}$ where the index is understood
modulo $n$.

Define $u_0=x_{\frac{n-1}{2}}-x_{\frac{n+1}{2}}$,
$t=x_{\frac{n-1}{2}}$, $u_i=\sigma^i(u_0)$ for $0\le i\le n-1$.

It follows that $\sum_{0\le i\le n-1} u_i=0$ and
$\{u_1,u_2,\ldots,u_{n-1},t\}$ is a $\bm{Z}$-basis of $M_+$ with
$\sigma$ and $\tau$ acting by
\begin{align*}
\sigma :{}& u_1\mapsto u_2\mapsto \cdots \mapsto u_{n-1}\mapsto u_0=-(u_1+u_2+\cdots+u_{n-1}), \\
& t\mapsto t+u_1+u_2+\cdots+u_{n-1}, \\
\tau :{}& u_i\leftrightarrow-u_{n-i},~ t\mapsto
t+u_1+u_2+\cdots+u_{n-1}.
\end{align*}

Note that $\sum_{1\le i\le n-1} \bm{Z}\cdot u_i\simeq N_-$ and
$M_+/(\sum_{1\le i\le n-1} \bm{Z}\cdot u_i)\simeq \bm{Z}$. Hence
we get the sequence $0\to N_-\to M_+\to \bm{Z} \to 0$.

We claim that this sequence doesn't split. Otherwise, $M_+ \simeq N_- \oplus \bm{Z}$. Since both $M_+$ and $\bm{Z}$ are permutation lattices, it follows that $H^1(S, N_-)=0$ for any subgroup $S$ of $G$. On the other hand, it is not difficult to see that $H^1(\langle \sigma \rangle, N_-)=\bm{Z}/n \bm{Z}$. Thus we find a contradiction.

\begin{Case}{2} $M_-$. \end{Case}

The proof is similar to Case 1. Choose a $\bm{Z}$-basis $\{x_i:
0\le i\le n-1\}$ of $M_-$ with $\sigma: x_i\mapsto x_{i+1}$,
$\tau: x_i\mapsto -x_{n-i}$.

Define $u_0=x_{\frac{n-1}{2}}-x_{\frac{n+1}{2}}$,
$t=x_{\frac{n-1}{2}}$, $u_i=\sigma^i (u_0)$ for $0\le i\le n-1$.
We find that
\begin{align*}
\sigma :{}& u_1\mapsto u_2\mapsto \cdots \mapsto u_{n-1}\mapsto -(u_1+u_2+\cdots+u_{n-1}), \\
& t\mapsto t+u_1+u_2+\cdots+u_{n-1}, \\
\tau :{}& u_i\leftrightarrow u_{n-i},~ t\mapsto
-t-u_1-u_2-\cdots-u_{n-1}.
\end{align*}

Thus $\sum_{0\le i\le n-1} \bm{Z}\cdot u_i \simeq N_+$ and
$M_-/(\sum_{1\le i\le n-1} \bm{Z}u_i)\simeq \bm{Z}_-$.

Again, the sequence $0\to N_+\to M_-\to \bm{Z}_-\to 0$ doesn't
split. Otherwise, there is some element $s\in M_-$ satisfying that $\sigma(s)=s$, $\tau(s)=-s$ and $M_-=N_+ \oplus \bm{Z}s$. A straightforward computation shows that it is necessary that $s$ is of the form $b(u_1+2u_2+\cdots +iu_i+\cdots +(n-1)u_{n-1}+nt)$ where $b$ is some integer. In that case, $N_+ + \bm{Z}s$ is a proper sublattice of $M_-$. Hence we get a contradiction.

\end{proof}

\begin{lemma} \label{l4.5}
Let $N_+$, $N_-$, $\widetilde{M}_+$, $\widetilde{M}_-$, $\bm{Z}$, $\bm{Z}_-$ be
$G$-lattices with
$G=\langle\sigma,\tau:\sigma^n=\tau^2=1,\tau\sigma\tau^{-1}=\sigma^{-1}\rangle
\simeq D_n$ in Definition \ref{d3.1}, Definition \ref{d3.2} and  Definition \ref{d3.3} where $n$ is an odd integer. Then there are non-split
exact sequences of $G$-lattices $0\to N_-\to
\widetilde{M}_+\to \bm{Z}[G/\langle\sigma\rangle]\to 0$ and $0\to N_+\to \widetilde{M}_- \to
\bm{Z}[G/\langle\sigma\rangle]\to 0$.

Moreover, the sequences $0\to N_-\to M_+\to \bm{Z}\to 0$ and $0\to N_+\to M_-\to \bm{Z}_-\to 0$ in Lemma \ref{l4.4} are the pull-backs of the sequences  $0\to N_-\to
\widetilde{M}_+\to \bm{Z}[G/\langle\sigma\rangle]\to 0$ and $0\to N_+\to \widetilde{M}_- \to
\bm{Z}[G/\langle\sigma\rangle]\to 0$ through the map $\alpha: \bm{Z} \to \bm{Z}[G/\langle\sigma\rangle]$ and the map $\alpha_-: \bm{Z_-} \to \bm{Z}[G/\langle\sigma\rangle]$ respectively.
\end{lemma}

\begin{proof}
\begin{Case}{1} $\widetilde{M}_+$. \end{Case}

We adopt the same notations $x_0,x_1,\ldots,x_{n-1},u_1,\ldots,u_{n-1}$ in the proof of Lemma \ref{l4.4}.
Write $\widetilde{M}_+=(\bigoplus_{0\le i\le n-1} \bm{Z}\cdot x_i)\oplus \bm{Z}\cdot w$ with
\begin{align*}
\sigma &: x_i\mapsto x_{i+1}, ~ w\mapsto w, \\
\tau &: x_i \mapsto x_{n-i}, ~ w\mapsto -w+x_0+x_1+\cdots+x_{n-1}
\end{align*}
where the index is understood modulo $n$.

Define $u_0=x_{\frac{n-1}{2}} - x_{\frac{n+1}{2}}$, $t=x_{\frac{n-1}{2}}$, $u_i=\sigma^i (u_0)$ for $0\le i\le n-1$.

We claim that
\begin{align*}
\sum_{0\le i\le n-1} x_i=nt-(n-1)u_0-(n-2)u_1-\cdots-u_{n-2}.
\end{align*}

Since $x_{\frac{n-1}{2}}=t$, $x_{\frac{n+1}{2}}=t-u_0$, we find
that $x_{\frac{n+3}{2}}=x_{\frac{n+1}{2}}-u_1=t-u_0-u_1$. By
induction, we may find similar formulae for $x_{\frac{n+5}{2}},
\ldots,x_{n-1},x_0,\ldots,x_{\frac{n-3}{2}}$. In particular,
$x_{\frac{n-3}{2}}=t-u_0-u_1-\cdots -u_{n-2}$. Thus the formula of
$\sum_{0\le i\le n-1} x_i$ is found.

Note that $\{u_1,\ldots,u_{n-1},t,w\}$ is a $\bm{Z}$-basis of $\widetilde{M}_+$ and
\begin{align*}
\sigma &: u_1\mapsto u_2\mapsto \cdots \mapsto u_{n-1}\mapsto -(u_1+u_2+\cdots+u_{n-1}), ~ t\mapsto t+\sum_{1\le i\le n-1} u_i,~ w\mapsto w, \\
\tau &: u_i\mapsto -u_{n-i},~ t\mapsto t+\sum_{1\le i\le n-1} u_i,~ w\mapsto -w+u_1+2u_2+\cdots+(n-1)u_{n-1}+nt.
\end{align*}

Since $\sigma$ acts trivially on $\widetilde{M}_+/N_- =\bm{Z}\bar{t}\oplus\bm{Z}\bar{w}$, the $G$-lattice $\widetilde{M}_+/N_-$ may be regarded as a $\bar{G}$-lattice where $\bar{G}=G/\langle \sigma \rangle \simeq \langle \tau \rangle \simeq C_2$.

Define $w_0=-\frac{n-1}{2}t+w$, $w_1=\frac{n+1}{2}t-w$.
Then $\{u_1,\ldots,u_{n-1},w_0,w_1\}$ is also a $\bm{Z}$-basis of $\widetilde{M}_+$ and $\widetilde{M}_+/N_- =\bm{Z}\bar{w}_0\oplus\bm{Z}\bar{w}_1$. It is routine to verify that $\tau: \bar{w}_0\mapsto \bar{w}_1 \mapsto \bar{w}_0$. Thus the $\bar{G}$-lattice $\widetilde{M}_+/N_-$ is isomorphic to $\bm{Z}[\bar{G}]$. The exact sequence $0\to N_-\to
\widetilde{M}_+\to \bm{Z}[G/\langle\sigma\rangle]\to 0$ is found.

\medskip
It is easy to verify that $0\to N_-\to M_+\to \bm{Z}\to 0$ is the pull-back of the sequences $0\to N_-\to
\widetilde{M}_+\to \bm{Z}[G/\langle\sigma\rangle]\to 0$ through the map $\alpha: \bm{Z} \to \bm{Z}[G/\langle\sigma\rangle]$ defined by $\alpha(w)=1+\tau \in \bm{Z}[G/\langle\sigma\rangle]$ where $\bm{Z}=\bm{Z}\cdot w$.

Since the pull-back sequence $0\to N_-\to M_+\to \bm{Z}\to 0$ doesn't split, the source sequence $0\to N_-\to
\widetilde{M}_+\to \bm{Z}[G/\langle\sigma\rangle]\to 0$ will not split (consider the map of the extension groups induced by $\alpha$).

\medskip
\begin{Case}{2} $\widetilde{M}_-$. \end{Case}

The proof is similar.
We adopt the notations $x_0,x_1,\ldots,x_{n-1},u_1,\ldots,u_{n-1}$ in the proof of Case 1.
Write $\widetilde{M}_- =(\bigoplus_{0\le i\le n-1} \bm{Z}\cdot x_i)\oplus \bm{Z}\cdot w$ such that
\begin{align*}
\sigma &: x_i\mapsto x_{i+1},~ w\mapsto w, \\
\tau &: x_i\mapsto -x_{n-i},~ w\mapsto w-\sum_{0\le i\le n-1} x_i
\end{align*}
where $0\le i\le n-1$ and the index is understood modulo $n$.

Define $u_0=x_{\frac{n-1}{2}}-x_{\frac{n+1}{2}}$, $t=x_{\frac{n-1}{2}}$, $u_i=\sigma^i(u_0)$ for $0\le i\le n-1$.
Then $\{u_1,\ldots,u_{n-1},t,w\}$ is a $\bm{Z}$-basis of $\widetilde{M}_-$ and
\begin{align*}
\sigma:{}& u_1\mapsto u_2\mapsto\cdots\mapsto u_{n-1}\mapsto -(u_1{+}\cdots{+}u_{n-1}),\\
& t\mapsto t+u_1+u_2+\cdots+u_{n-1},\, w\mapsto w, \\
\tau:{}& u_i\mapsto u_{n-i},~ t\mapsto -t-\left(\sum_{1\le i\le n-1} u_i\right),\\
& w\mapsto w-u_1-2u_2-\cdots -(n-1)u_{n-1}-nt.
\end{align*}

Consider $\widetilde{M}_-/N_+$ and define $w_0=\frac{n-1}{2}t-w$, $w_1=\frac{n+1}{2}t-w$. The remaining proof is similar to {Case}{1} and is omitted.
\end{proof}

\begin{lemma} \label{l4.6}
Let $N_+$, $N_-$ be $G$-lattices in Definition \ref{d3.2} with
$G=\langle\sigma,\tau:\sigma^n=\tau^2=1,\tau\sigma\tau^{-1}=\sigma^{-1}\rangle
\simeq D_n$ where $n$ is an odd integer. Then there is a non-split
exact sequences of $G$-lattices $0\to N_+ \oplus N_- \to
\bm{Z}[G]\to \bm{Z}[G/\langle\sigma\rangle] \to 0$.
\end{lemma}

\begin{proof}
This lemma was proved by Lee for the case when $n=p$ is an odd
prime number in (i) of Case 1 of \cite[pages 222--224]{Le}. There
was also a remark in the first paragraph of \cite[page 229,
Section 4]{Le}.

Here is a proof for the general case when $n$ is an odd integer.

\bigskip
Step 1. Let $\{\sigma^i,\sigma^i \tau: 0\le i\le n-1\}$ be a
$\bm{Z}$-basis of $\bm{Z}[G]$.

Let $\{t_0,t_1\}$ be a $\bm{Z}$-basis of $\bm{Z}[G/\langle \sigma \rangle]$ with
$\sigma(t_i)=t_i$, $\tau:t_0\leftrightarrow t_1$.

Define a $G$-lattice surjection $\varphi:\bm{Z}[G]\to \bm{Z}[G/\langle \sigma \rangle]$
by $\varphi(\sigma^i)=t_0$, $\varphi(\sigma^i \tau)=t_1$. Define a
$G$-lattice $M$ by $M=\fn{Ker}(\varphi)$. We will prove that $M
\simeq N_+ \oplus N_-$ (note that $\bm{Z}[G]$ is indecomposable).

Define $u_i, v_i \in M$ as follows. Define $u_0=
\sigma^{(n-1)/2}-\sigma^{(n+1)/2}$, $v_0=\sigma^{(n+1)/2}
\tau-\sigma^{(n-1)/2} \tau$, and $u_i= \sigma^i(u_0)$, $v_i=
\sigma^i(v_0)$ for $0 \le i \le n-1$.

It follows that $\sum_{0 \le i \le n-1} u_i = \sum_{0 \le i \le
n-1} v_i=0$, and $\{u_i,v_i : 1\le i\le n-1\}$ is a $\bm{Z}$-basis
of $M$. Moreover, it is easy to see that $\sigma: u_i \mapsto
u_{i+1},v_i \mapsto v_{i+1}$, $\tau: u_i \mapsto v_{n-i},v_i
\mapsto u_{n-i}$ where the index is understood modulo $n$.

\bigskip
Step 2.

Define $x_i=u_i+v_i$, $y_i=u_{i-1}-v_{i+1}$ where $0\le i \le
n-1$. Clearly $\sum_{0 \le i \le n-1}x_i=\sum_{0 \le i \le
n-1}y_i=0$. We claim that $\{x_i, y_i: 1 \le i \le n-1 \}$ is a
$\bm{Z}$-basis of $M$.

Assume the above claim. Define $M_1=\oplus_{1 \le i \le n-1}
\bm{Z} \cdot x_i$, $M_2=\oplus_{1 \le i \le n-1} \bm{Z} \cdot
y_i$. It is easy to verify that $M_1 \simeq N_+$ and $M_2 \simeq
N_-$. Hence the proof that $M \simeq N_+ \oplus N_-$ is finished.

\bigskip
Step 3.

We will prove that $\{x_i, y_i: 1 \le i \le n-1 \}$ is a
$\bm{Z}$-basis of $M$.

Let $Q$ be the coefficient matrix of $x_1, x_2,
\ldots,x_{n-1},y_1, \ldots,y_{n-1}$ with respect to the
$\bm{Z}$-basis $u_1, u_2, \ldots,u_{n-1},v_1, \ldots,v_{n-1}$. For
the sake of visual convenience, we will consider the matrix $P$
which is the transpose of $Q$. We will show that det$(P)= 1$.


The matrix $P$ is defined as
\[
P=\left(\begin{array}{@{}cccc;{3pt/2pt}ccccc@{}}
1 &  &  &  & 1 &  &  & \\
 & \ddots &  &  &  &  \ddots & & \\
 &  & \ddots &  &  &  & \ddots & \\
 &  &  & 1 &  &  & & 1 \\\hdashline[3pt/2pt]
-1 & \cdots & -1 & -1 & 0 & -1 & & \\
1 &  &  & 0 & \vdots &  & \ddots & \\
 &  \ddots &  & \vdots & 0 &  &  & -1 \\
 &   & 1 &  0 & 1 & 1 &  \cdots & 1
\end{array}\right).
\]

For examples, when $n=3, 5$, it is of the form

\begin{align*}
P&=\left(\begin{array}{@{}cc;{3pt/2pt}ccc@{}}
1 & 0 & 1 & 0\\
0 & 1 & 0 & 1\\\hdashline[3pt/2pt]
-1 & -1 & 0 & -1\\
1 &  0 & 1 & 1
\end{array}\right),\\
P&=\left(\begin{array}{@{}cccc;{3pt/2pt}ccccc@{}}
1 & 0 & 0 & 0 & 1 & 0 & 0 & 0\\
0 & 1 & 0 & 0 & 0 & 1 & 0 & 0\\
0 & 0 & 1 & 0 & 0 & 0 & 1 & 0\\
0 & 0 & 0 & 1 & 0 & 0 & 0 & 1 \\\hdashline[3pt/2pt]
-1 & -1 & -1 & -1 & 0 & -1 & 0 & 0\\
1 & 0 & 0 & 0 & 0 & 0 & -1 & 0\\
0 & 1 & 0 & 0 & 0 & 0 & 0 & -1 \\
0 & 0 & 1 &  0 & 1 & 1 &  1 & 1
\end{array}\right).
\end{align*}

In the case $n=3,5$, it is routine to show that det$(P)= 1$. When
$n \ge 7$, we will apply column operations on the matrix $P$ and
then expand the determinant along a row. Thus we are reduced to
matrices of smaller size.

For a given matrix, we denote by $C(i)$ 
its $i$-th column.
When we say that, apply $C(i)+C(1)$ on the $i$-th column, we mean
the column operation by adding the 1st column to the $i$-th
column.

\bigskip
Step 4.

We will prove det$(P)= 1$ where $P$ is the $(2n-2)\times(2n-2)$
integral matrix defined in Step 3. Suppose $n \ge 7$.

Apply column operations on the matrix $P$. On the $(n+i)$-th
column where $0 \le i \le n-2$, apply $C(n+i)-C(i+1)$.

Thus all the entries of the right upper part of the resulting
matrix vanish. We get det$(P)=$ det$(P_0)$ where $P_0$ is an
$(n-1)\times(n-1)$ integral matrix defined as

\begin{align*}
P_0&=\left(\begin{array}{@{}cccccccc@{}}
1 & 0 & 1 & 1 & 1 & 1 & \cdots & 1\\
-1 & 0 & -1 & & & & &  \\
& -1 & 0 & -1 & & & &  \\
& & -1 & 0 & -1 & & &  \\
& & & -1 & 0 & -1 & &  \\
& & &  & \ddots & \ddots & \ddots & \\
& & & &  & -1 & 0 & -1 \\
1 & 1 & 1 & \cdots& 1 & 1 & 0 & 1
\end{array}\right).
\end{align*}

\bigskip
Step 5.

Apply column operations on $P_0$. On the 3rd column, apply
$C(3)-C(1)$. Then, on the 4th column, apply $C(4)-C(2)$.

In the resulting matrix, each of the 2nd row and the 3rd row have
only one non-zero entry.

Thus det$(P_0) = $ det$(P_1)$ where $P_1$ is an $(n-3)\times(n-3)$
integral matrix defined as

\begin{align*}
P_1&=\left(\begin{array}{@{}cccccccc@{}}
0 & 1 & 1 & 1 & 1 & \cdots & 1\\
-1 & 0 & -1 & & & & &  \\
& -1 & 0 & -1 & & & &  \\
& & -1 & 0 & -1 & & &  \\
 & &  & \ddots & \ddots & \ddots & \\
& & & &  -1 & 0 & -1 \\
0 & 0 & 1 & \cdots& 1 & 0 & 1
\end{array}\right).
\end{align*}

\bigskip
Step 6.

Apply column operations on $P_1$. On the 3rd column, apply
$C(3)-C(1)$. On the 4th column, apply $C(4)-C(2)$.

Then expand the determinant along the 2nd row and the 3rd row. We
get det$(P_1) = $ det$(P_2)$ where $P_2$ is an $(n-5)\times(n-5)$
integral matrix defined as

\begin{align*}
P_2&=\left(\begin{array}{@{}ccccc@{}}
1 & 0 & 1 & \cdots & 1 \\
-1 & 0 & -1 &  &  \\
 & \ddots & \ddots & \ddots & \\
 &  & -1 & 0 & -1 \\
1 & \cdots & 1 & 0 & 1
\end{array}\right).
\end{align*}

But the matrix $P_2$ looks the same as $P_0$ except the size.
Done.

\end{proof}

The following lemma is suggested by the referee.

\begin{lemma} \label{l3.9}
Let $I_G:=\fn{Ker}\{\bm{Z}[G]\xrightarrow{\varepsilon} \bm{Z}\}$ be the augmentation ideal of $\bm{Z}[G]$ where
$\varepsilon(\sum_{g\in G} n_g\cdot g)=\sum_{g\in G} n_g$ with $G=\langle\sigma,\tau:\sigma^n=\tau^2=1,\tau\sigma\tau^{-1}=\sigma^{-1}\rangle
\simeq D_n$ and $n$ an odd integer. Let $M_+$, $M_-$, $N_-$, $\widetilde{M}_+$, $\widetilde{M}_-$ be
$G$-lattices in Definition \ref{d3.1}, Definition \ref{d3.2} and  Definition \ref{d3.3}. Then $I_G \simeq M_- \oplus N_-$ and there are exact sequences of $G$-lattices $0\to M_+\to \bm{Z}[G]\to M_- \to 0$, $0\to \widetilde{M}_+\to \bm{Z}[G]\to N_- \to 0$, $0\to M_+ \oplus \widetilde{M}_+\to \bm{Z}[G]^{(2)}\to I_G \to 0$
\end{lemma}

\begin{proof}
Step 1. Note that $G$ is a Frobenius group. Thus we may use the proof of Theorem 8.5 (iii) $\Rightarrow$ (i) in \cite[pages 57-58]{Gr1}. More explicitly, write $H=\langle \tau \rangle$. By Step 1 and Step 2 of \cite[pages 57-58]{Gr1}, the $G$-lattice $I_G$ is isomorphic to the direct sum of the $G$-lattice $I_{\bm{Z}[G/H]}=\sum_{1 \le i \le n-1}\bm{Z}\cdot (\sigma^{i-1}H-\sigma^iH)$ and the $G$-lattice $\sum_{0 \le i \le n-1}\bm{Z}\cdot \sigma^i(1-\tau)$.

On the other hand, it is easy to verify that the $G$-lattice $I_{\bm{Z}[G/H]}$ is isomorphic to $N_-$ with $\sigma^{i-1}H-\sigma^iH$ corresponding to the basis element in Definition \ref{d3.2}. Similarly, the $G$-lattice $\sum_{0 \le i \le n-1}\bm{Z}\cdot \sigma^i(1-\tau)$ is isomorphic to $M_-$ with $\sigma^i(1-\tau)$ corresponding to the basis element $x_i$ in the proof of Lemma \ref{l4.4}.

\bigskip
Step 2. For the exact sequence $0\to M_+\to \bm{Z}[G]\to M_- \to 0$, note that the $G$-sublattice $M:=\sum_{0 \le i \le n-1}\bm{Z}\cdot \sigma^i(1+\tau)$ of $\bm{Z}[G]$ is isomorphic to $M_+$ with $\sigma^i(1+\tau)$ corresponding to the basis element in Definition \ref{d3.1}. Choose $\sigma^i$ (where $0 \le i \le n-1$) to be a basis for $\bm{Z}[G]/M$. Then it is easy to verify that $\bm{Z}[G]/M$ is isomorphic to $M_-$ with $\sigma^i$ corresponding to the basis element in Definition \ref{d3.1}. Hence the result.

It remains to find the sequence $0\to \widetilde{M}_+\to \bm{Z}[G]\to N_- \to 0$. Define a sublattice $\widetilde{M}:=\sum_{0 \le i \le n-1}\bm{Z}\cdot \sigma^i(1+\tau)+\bm{Z}\cdot \sum_{0\le i \le n-1}\sigma^i\tau$ of $\bm{Z}[G]$. The lattice $\widetilde{M}$ is isomorphic to $\widetilde{M}_+$ with $\sigma^i(1+\tau)$ corresponding to the basis element $w_i$ and $\sum_{0\le i \le n-1}\sigma^i\tau$ corresponding to $w$ where $w_i$ and $w$ are the basis of $\widetilde{M}_+$ in Definition \ref{d3.3}. Choose $\sigma^i$ (where $1 \le i \le n-1$) to be a basis for $\bm{Z}[G]/\widetilde{M}$. It is not difficult to show that $\bm{Z}[G]/\widetilde{M}$ is isomorphic to $N_-$ with $\sigma^i$ corresponding to the basis element in Definition \ref{d3.2}.
\end{proof}

\section{The relation modules}

Recall the relation module $R^{ab}$ in Section 1.

\begin{lemma} \label{l5.1}
Let $G$ be a finite group, $1\to R\to F\to G\to 1$ be a free presentation of $G$ where $F$ is a free group of finite rank.
Then $R$ is a free group of finite rank and $R^{ab}$ is a faithful $G$-lattice,
i.e.\ if $g\in G$ and $g \cdot x=x$ for any $x\in R^{ab}$, then $g=1$.
\end{lemma}

\begin{proof}
$R$ is a free group of finite rank by the Nielsen-Schreier Theorem
\cite[page 36]{Kur}; later we will exhibit a free generating set
of $R$ as indicated in \cite[Theorem 8.1]{Ma}. Since $R$ is free,
$R^{ab}$ is a free abelian group of finite rank. Thus $R^{ab}$ is
a $G$-lattice. It is a faithful $G$-lattice by a theorem of Gasch\"utz (see
\cite[page 8]{Gr1}).
\end{proof}

\begin{lemma} \label{l5.2}
Let $G$ be a finite group, $1\to R\to F\to G\to 1$ be a free
presentation of $G$ where $F$ is a free group of rank $d$. Let
$I_G:=\fn{Ker}\{\bm{Z}[G]\xrightarrow{\varepsilon} \bm{Z}\}$ be
the augmentation ideal of $\bm{Z}[G]$ where
$\varepsilon(\sum_{g\in G} n_g\cdot g)=\sum_{g\in G} n_g$.
\begin{enumerate}
\item[{\rm (1)}] There is a short exact sequence of $G$-lattices
$0\to R^{ab} \to \bm{Z}[G]^{(d)} \to I_G \to 0$.  \item[{\rm (2)}]
Write $I_G^{\otimes 2}:=I_G\otimes_{\bm{Z}} I_G$. Then $R^{ab}
\oplus \bm{Z}[G]^{(|G|-1)} \simeq I_G^{\otimes 2} \oplus
\bm{Z}[G]^{(d)}$. \item[{\rm (3)}] Both $R^{ab}$ and $I_G^{\otimes 2}$ are coflabby and $0\to (I_G)^0 \to \bm{Z}[G]^{(|G|-1)} \to (I_G^{\otimes 2})^0\to 0$ is a flabby resolution of $(I_G)^0$.
\end{enumerate}
\end{lemma}

\begin{proof}
(1) By \cite[page 199, Corollary 6.4]{HS}, we have a short exact sequence $0\to R^{ab} \to \bm{Z}[G]\otimes_{\bm{Z}[F]} I_F \to I_G\to 0$.
Since $I_F$ is a free $\bm{Z}[F]$-module of rank $d$ by \cite[page 196, Theorem 5.5]{HS},
we find that $\bm{Z}[G]\otimes_{\bm{Z}[F]} I_F \simeq \bm{Z}[G]^{(d)}$.

(2) Tensor $0\to I_G \to \bm{Z}[G]\to \bm{Z}\to 0$ with $I_G$ over $\bm{Z}$.
We get $0\to I_G^{\otimes 2}\to \bm{Z}[G] \otimes_{\bm{Z}} I_G \to I_G\to 0$.
For a $\bm{Z}[G]$-module $M$, write $M_0$ the underlying abelian group of $M$;
thus $M_0$ becomes a trivial $\bm{Z}[G]$-module.
By \cite[page 212, Lemma 11.7, Corollary 11.8]{HS}, the $\bm{Z}[G]$-modules $\bm{Z}[G]\otimes_{\bm{Z}} I_G$ and $\bm{Z}[G]\otimes_{\bm{Z}} (I_G)_0$ are isomorphic.
Hence $\bm{Z}[G]\otimes_{\bm{Z}} I_G \simeq \bm{Z}[G]^{(|G|-1)}$.

In summary, we have two short exact sequences $0\to R^{ab}\to \bm{Z}[G]^{(d)} \to I_G \to 0$ and $0\to I_G^{\otimes 2} \to \bm{Z}[G]^{(|G|-1)} \to I_G\to 0$.
Apply Schanuel's Lemma.

(3) It suffices to show that $H^1(S, I_G^{\otimes 2})=0$ for any subgroup $S$ of
$G$. Use the exact sequences $0\to I_G^{\otimes 2} \to \bm{Z}[G]^{(|G|-1)} \to I_G\to 0$ (in (2)) and $0\to I_G\to \bm{Z}[G]\to \bm{Z}\to 0$. We find that $H^1(S,I_G^{\otimes 2})\simeq \widehat{H}^0
(S,I_G)\simeq H^{-1}(S,\bm{Z})=0$.
\end{proof}

\begin{remark}
{}From the above lemma, we may determine the rank of the
free abelian group $R^{ab}$. This is also the rank of $R$ as a
free group; thus it supplies a proof of Schreier's Theorem \cite[page
36]{Kur} by homological algebra.
\end{remark}

\begin{lemma} \label{l5.3}
Let $G$ be a finite group,
$1\to R_1 \to F_1 \to G\to 1$ and $1\to R_2 \to F_2 \to G\to 1$ be two free presentations of $G$ where $F_i$ is a free group of rank $d_i$ for $i=1,2$. Then $R_1^{ab}\oplus \bm{Z}[G]^{(d_2)} \simeq R_2^{ab}\oplus \bm{Z}[G]^{(d_1)}$.

\end{lemma}

\begin{proof}
By Lemma \ref{l5.2} we have exact sequences of $G$-lattices $0\to R_i^{ab}\to \bm{Z}[G]^{(d_i)}\to I_G\to 0$ for $i=1,2$.
Apply Schanuel's Lemma.
\end{proof}

\begin{lemma} \label{l5.5}
Let $G=D_n$ be the dihedral group of order $2n$ where $n\ge 2$.
Let $1\to R\to F\to G\to 1$ be a free presentation where $F$ is a free group of finite rank.
\begin{enumerate}
\item[{\rm (1)}]
If $n$ is odd, then $R^{ab}$ is an invertible lattice.
\item[{\rm (2)}]
If $n$ is even, then $R^{ab}$ is coflabby, but is not flabby.
\end{enumerate}
\end{lemma}

\begin{proof}
(1) Suppose $n$ is odd. Then all the $p$-Sylow subgroups of $D_n$
are cyclic groups. Hence we may apply Theorem \ref{t2.6}. Thus
$[I_G^0]^{fl}$ is invertible.

On the other hand, from Lemma \ref{l5.2}, we get $0\to R^{ab}\to \bm{Z}[G]^{(d)} \to I_G\to 0$.
Taking the dual lattices, we get $0\to I_G^0\to \bm{Z}[G]^{(d)}\to (R^{ab})^0\to 0$.
By Lemma \ref{l5.2}, $R^{ab}$ is coflabby.
Hence $(R^{ab})^0$ is flabby.
Thus $0\to I_G^0 \to \bm{Z}[G]^{(d)}\to (R^{ab})^0\to 0$ is a flabby resolution of $I_G^0$,
i.e.\ $[I_G^0]^{fl}=[(R^{ab})^0]$.

Since $[I_G^0]^{fl}$ is invertible, it follows that $[(R^{ab})^0]$ is invertible.
Hence there is some permutation $G$-lattice $P$ such that $(R^{ab})^0 \oplus P$ is an invertible lattice.
Thus $(R^{ab})^0\oplus P\oplus M=Q$ for some $G$-lattice $M$ and some permutation $G$-lattice $Q$.
Taking dual, we find that $R^{ab}$ is invertible.

(2) By Lemma \ref{l5.2}, $R^{ab}$ is coflabby. It remains to show that
$R^{ab}$ is not flabby. Since $n$ is even, the subgroup
$S=\langle\tau,\sigma^{n/2}\rangle$ exists and is isomorphic to
$C_2 \times C_2$. We will show that $H^{-1}(S,R^{ab})\ne 0$.

{}From the exact sequences $0\to R^{ab}\to \bm{Z}[G]^{(d)}\to I_G
\to 0$ and $0\to I_G\to \bm{Z}[G]\to \bm{Z}\to 0$, we find that
$H^{-1}(S,R^{ab})\simeq H^{-2}(S,I_G)\simeq
H^{-3}(S,\bm{Z})=H_2(S,\bm{Z})\simeq \bm{Z}/2\bm{Z}$ by \cite[page
223, Theorem 15.2]{HS}, because $H_1(C_2,\bm{Z})\simeq
\bm{Z}/2\bm{Z}$, $H_2(C_2,\bm{Z})=0$.
\end{proof}

An immediate consequence of Endo and Miyata's Theorem, i.e. Theorem \ref{t2.7}, is the following.

\begin{theorem} \label{t5.8}
Let $R^{ab}$ be the relation module associated to the free resolution $1\to R\to F\xrightarrow{\varepsilon} D_n\to 1$ in Definition \ref{d1.1} where $n \ge 2$. If $K/k$ is a Galois extension with Galois group $D_n$, then $K(R^{ab})^{D_n}$ is stably rational over $k$ if and only if $n$ is odd. If $k$ is an infinite field and $n$ is even, then $K(R^{ab})^{D_n}$ is not retract rational over $k$.
\end{theorem}

\begin{proof}
By Theorem \ref{t2.7}, $[I_{D_n}^0]^{fl}=0$ if and only if $n$ is an odd integer. On the other hand, $[I_{D_n}^0]^{fl}=[(I_{D_n}^{\otimes 2})^0]=[(R^{ab})^0]$ by Lemma \ref{l5.2}. Note that $[(R^{ab})^0]=0 \Leftrightarrow [(R^{ab})]=0 \Leftrightarrow [(R^{ab})]^{fl}=0$, because the flabby resolution $0 \rightarrow R^{ab} \rightarrow P \rightarrow E \rightarrow 0$ splits whenever $E$ is invertible (see the following paragraph). Apply Theorem \ref{t2.5}. The proof for the stable rationality is finished.

Now suppose that $k$ is an infinite field and $n$ is even. Choose a flabby resolution $0 \rightarrow R^{ab} \rightarrow P \rightarrow E \rightarrow 0$ for $R^{ab}$ where $P$ is permutation and $E$ is flabby. We claim that $[E] \in F_{D_n}$ is not invertible. Suppose not. Then we may assume that $E$ is an invertible lattice without loss of generality. Since $R^{ab}$ is coflabby by Lemma \ref{l5.2}, the above exact sequence splits because of \cite[Proposition 1.2]{Len}. It leads to the fact that $P \simeq R^{ab} \oplus E$, i.e. $R^{ab}$ is invertible. This is impossible by Lemma \ref{l5.5}. Because $[E] \in F_{D_n}$ is not invertible, we find that $K(R^{ab})^{D_n}$ is not retract rational over $k$ by Theorem \ref{t2.5}.
\end{proof}

The following theorem was communicated to us by the referee. We will give another decomposition of $R^{ab}$ (when $n$ is odd) in Theorem \ref{t5.7}.

\begin{theorem} \label{t5.12}
Let $R^{ab}$ be the relation module of the free presentation of
$D_n$ in Definition \ref{d1.1}. If $n$ is an odd integer $\ge 3$,
then $R^{ab}\simeq M\oplus \widetilde{M}$ where $M$ and $M_+$ (resp. $\widetilde{M}$ and
$\widetilde{M}_+$) belong to the same genus.
\end{theorem}

\begin{proof}
Write $G=D_n$ where $n$ is odd.

{}From Lemma \ref{l3.9} we have a short exact sequence $0\to M_+ \oplus \widetilde{M}_+\to \bm{Z}[G]^{(2)}\to I_G \to 0$. From Lemma \ref{l5.2} we have another short exact sequence $0\to R^{ab} \to \bm{Z}[G]^{(2)} \to I_G \to 0$. By Schanuel's Lemma, we find that $(M_+ \oplus \widetilde{M}_+)\oplus \bm{Z}[G]^{(2)} \simeq R^{ab}\oplus \bm{Z}[G]^{(2)}$.

Let $(\bm{Z}[G]^{(2)})_{(G)}, (M_+)_{(G)}, \ldots$ be the localizations of $\bm{Z}[G]^{(2)}, M_+ \ldots$ with respect to the multiplicative closed set $\bm{Z} \setminus \cup_p p\bm{Z}$ where $p$ runs over all the prime divisors of $\mid G \mid$.

{}From $(M_+ \oplus \widetilde{M}_+)\oplus \bm{Z}[G]^{(2)} \simeq R^{ab}\oplus \bm{Z}[G]^{(2)}$, we have $(M_+)_{(G)}\oplus (\widetilde{M}_+)_{(G)}\oplus (\bm{Z}[G]^{(2)})_{(G)} \simeq (R^{ab})_{(G)}\oplus (\bm{Z}[G]^{(2)})_{(G)}$. By the semi-local cancellation \cite[page 19]{Gr1}, we obtain that $(M_+)_{(G)}\oplus (\widetilde{M}_+)_{(G)} \simeq (R^{ab})_{(G)}$. By \cite[page 26, Proposition 5.1]{Gr1} we find that $R^{ab} = M \oplus \widetilde{M}$ for some sublattices $M$ and $\widetilde{M}$ with the required property (see \cite[page 642]{CR1} for the definition of a genus).
\end{proof}

\bigskip
Now we will turn to finding the free generators of $R$ in a free presentation $1\to R\to F\xrightarrow{\varepsilon} G\to 1$
where $G$ is a finite group and $F$ is a free group of rank $d$.

It is known that $R$ is a free group of rank $1+(d-1)\cdot |G|$ by
Schreier's Theorem \cite[page 36]{Kur}. Let $s_1,s_2,\ldots,s_d$
be the free generators of $F$. Then the free generators of $R$ can
be found as follows (see \cite[pages 205--207]{Ma}).

We choose a Schreier system $\Sigma$. $\Sigma$ is a subset of $F$
such that $1\in\Sigma$, $|\Sigma\cap Rx|=1$ for any $x\in F$, and
some other properties are satisfied (see \cite[page 205]{Ma}).

Define a function $\Phi:F\to \Sigma$ such that
$\varepsilon(x)=\varepsilon(\Phi(x))$ for any $x\in F$.

Then the free generators of $R$ can be chosen as
$us_i\Phi(us_i)^{-1}$ where $u\in \Sigma$, $1\le i\le d$ and
$us_i\Phi(us_i)^{-1}\ne 1$ (see \cite[p.206, Theorem 8.1]{Ma}).

\medskip
In Definition \ref{d1.1} we consider a particular free
presentation $1\to R\to F\xrightarrow{\varepsilon} D_n\to 1$ where
$D_n=\langle
\sigma,\tau:\sigma^n=\tau^2=1,\tau\sigma\tau^{-1}=\sigma^{-1}\rangle$,
$F=\langle s_1,s_2 \rangle$ is the free group of rank two, and
$\varepsilon(s_1)=\sigma$, $\varepsilon(s_2)=\tau$. It is routine
to verify that $\Sigma=\{s_1^is_2^j:0\le i\le n-1$, $0\le j\le
1\}$ is a Schreier system of this free presentation. Hence the
free generators of $R$ can be chosen as
\begin{equation}
s_1^n,~s_2s_1s_2^{-1}s_1^{-(n-1)},~ s_1^is_2s_1s_2^{-1}s_1^{-(i-1)} ~ (1\le i\le n-1),~s_1^i s_2^2 s_1^{-i} ~(0\le i\le n-1). \label{q:1}
\end{equation}

We replace the generator $s_2s_1s_2^{-1}s_1^{-(n-1)}$ in \eqref{q:1} by $s_2s_1s_2^{-1}s_1^{-(n-1)}\cdot s_1^n=s_2s_1s_2^{-1}s_1$.
Hence we get the generators $a$, $b_i$, $c_i$ ($0\le i\le n-1$) of $R$ where
\begin{equation}
a=s_1^n,~ b_i=s_1^is_2s_1s_2^{-1}s_1^{-(i-1)},~ c_i=s_1^is_2^2s_1^{-i}. \label{q:2}
\end{equation}

Note that $F$ acts on $R$ by conjugation, i.e. for any $x\in F$, any $\alpha \in R$, $^x \alpha=x\cdot \alpha \cdot x^{-1}$.
It is not
difficult to verify that $s_1$ and $s_2$ act on these generators
by
\begin{align}
s_1:{}& a\mapsto a,~b_i\mapsto b_{i+1} \mbox{ if }0\le i\le n-2,~b_{n-1}\mapsto ab_0a^{-1}, \label{q:3} \\
& c_i\mapsto c_{i+1}\mbox{ if } 0\le i\le n-2,~ c_{n-1}\mapsto ac_0a^{-1}, \nonumber \\
s_2:{}& a\mapsto b_0a^{-1}b_{n-1}b_{n-2}\cdots b_1,~b_0\mapsto c_0c_1^{-1}b_1,~b_1\mapsto b_0a^{-1}c_{n-1}ac_0^{-1}, \label{q:4} \\
& b_i\mapsto b_0a^{-1}b_{n-1}b_{n-2}\cdots b_{n-i+1}c_{n-i}c_{n-i+1}^{-1}b_{n-i+2}^{-1}b_{n-i+3}^{-1}\cdots b_{n-1}^{-1}ab_0^{-1} \nonumber \\
& \mbox{(where }2\le i\le n-1), \nonumber \\
& c_0\mapsto c_0,~ c_1\mapsto (b_0a^{-1})c_{n-1}(b_0a^{-1})^{-1}, \nonumber \\
& c_i\mapsto (b_0a^{-1}b_{n-1}b_{n-2}\cdots b_{n-i+1})c_{n-i}(b_0a^{-1}b_{n-1}b_{n-2}\cdots b_{n-i+1})^{-1} \nonumber \\
& \mbox{(where }2\le i\le n-1). \nonumber
\end{align}
More explicitly, if $n$ is odd, the action of $s_2$ on
$b_2,b_3,\ldots,b_{n-1}$ is given by
\begin{align*}
b_2 &\mapsto b_0 a^{-1}b_{n-1}c_{n-2}c_{n-1}^{-1}ab_0^{-1}, \\
\vdots\; & \\
b_{\frac{n-1}{2}} &\mapsto b_0a^{-1}b_{n-1}b_{n-2}\cdots b_{\frac{n+3}{2}}c_{\frac{n+1}{2}}c_{\frac{n+3}{2}}^{-1}b_{\frac{n+5}{2}}^{-1}\cdots b_{n-1}^{-1}ab_0^{-1}, \\
b_{\frac{n+1}{2}} &\mapsto b_0a^{-1}b_{n-1}b_{n-2}\cdots b_{\frac{n+1}{2}}c_{\frac{n-1}{2}}c_{\frac{n+1}{2}}^{-1}b_{\frac{n+3}{2}}^{-1}\cdots b_{n-1}^{-1}ab_0^{-1}, \\
b_{\frac{n+3}{2}} &\mapsto b_0a^{-1}b_{n-1}b_{n-2}\cdots b_{\frac{n-1}{2}}c_{\frac{n-3}{2}}c_{\frac{n-1}{2}}^{-1}b_{\frac{n+1}{2}}^{-1}\cdots b_{n-1}^{-1}ab_0^{-1}, \\
\vdots\; & \\
b_{n-1} &\mapsto b_0a^{-1}b_{n-1}b_{n-2}\cdots b_2c_1c_2^{-1}b_3^{-1}\cdots b_{n-1}^{-1}ab_0^{-1}.
\end{align*}

\medskip
\begin{lemma} \label{l5.6}
Let $1\to R\to F\xrightarrow{\varepsilon} D_n\to 1$ be the free
presentation of
$D_n=\langle\sigma,\tau:\sigma^n=\tau^2=1,\tau\sigma\tau^{-1}=\sigma^{-1}\rangle$
in Definition \ref{d1.1}. Then $R^{ab}$ is a $D_n$-lattice with
$R^{ab}=\bm{Z}\cdot\bar{a}\oplus (\bigoplus_{0\le i\le
n-1}\bm{Z}\cdot \bar{b}_i)\oplus (\bigoplus_{0\le i\le n-1}
\bm{Z}\cdot \bar{c}_i)$ and the actions of $\sigma$ and $\tau$ on
$R^{ab}$ are given by
\begin{align*}
\sigma:{} & \bar{a}\mapsto \bar{a},~\bar{b}_0\mapsto\bar{b}_1\mapsto\cdots\mapsto\bar{b}_{n-1}\mapsto\bar{b}_0,~
\bar{c}_0\mapsto\bar{c}_1\mapsto\cdots\mapsto\bar{c}_{n-1}\mapsto c_0, \\
\tau:{} & \bar{a}\mapsto -\bar{a}+\sum_{0\le i\le n-1} \bar{b}_i,~\bar{b}_0\mapsto\bar{b}_1+\bar{c}_0-\bar{c}_1,~\bar{b}_1\mapsto\bar{b}_0+\bar{c}_{n-1}-\bar{c}_0, \\
& \bar{b}_i\mapsto\bar{b}_{n-i+1}+\bar{c}_{n-i}-\bar{c}_{n-i+1}\mbox{ for }2\le i\le n-1, \\
& \bar{c}_0\mapsto\bar{c}_0,~\bar{c}_i\mapsto\bar{c}_{n-i}\mbox{ for }1\le i\le n-1.
\end{align*}
\end{lemma}

\begin{proof}
Let $\bar{a}$, $\bar{b}_i$, $\bar{c}_i$ be the images of $a$,
$b_i$, $c_i$ in the canonical projection $R\to R^{ab}=R/[R,R]$
where $a$, $b_i$, $c_i$ are defined in \eqref{q:2}. The actions of
$\sigma$ and $\tau$ on $R^{ab}$ follow from Formulae \eqref{q:3}
and \eqref{q:4}.
\end{proof}

\bigskip
\begin{theorem} \label{t5.7}
Let $R^{ab}$ be the relation module of the free presentation of
$D_n$ in Definition \ref{d1.1}. If $n$ is an odd integer $\ge 3$,
then $R^{ab}\simeq M_+\oplus \widetilde{M}_+$ where $M_+$ and
$\widetilde{M}_+$ are the $D_n$-lattices in Definition \ref{d3.1}
and Definition \ref{d3.3}. Thus $R^{ab}\oplus \bm{Z}\simeq
\bm{Z}[D_n/\langle\sigma\rangle]\oplus
\bm{Z}[D_n/\langle\tau\rangle]^{(2)}$.
\end{theorem}

\begin{proof}
Once we prove $R^{ab}\simeq M_+\oplus \widetilde{M}_+$,
then identity $R^{ab}\oplus \bm{Z}\simeq \bm{Z}[D_n/\langle\sigma\rangle]\oplus\bm{Z}[D_n/\langle\tau\rangle]^{(2)}$ follows from Theorem \ref{t3.4}
because $M_+=\fn{Ind}_{\langle\tau\rangle}^{D_n}\bm{Z} \simeq \bm{Z}[D_n/\langle\tau\rangle]$ by definition.

We will give two proofs of $R^{ab}\simeq M_+\oplus \widetilde{M}_+$. The first proof relies on Lemma \ref{l5.6} and is computational. The second proof is suggested by the referee. It is based on a new technique, besides an adaptation of the method in \cite[pages 56-58]{Gr1} and \cite[page 199, Corollary 6.4]{GR} (see Step 2 and Step 3 of the second proof).

\medskip
The first proof ------------

\bigskip
Step 1. Recall the $n\times n$ matrices $A$ and $B$ in Definition
\ref{d3.1}. Define $I_n$ to be the identity matrix of size $n$,
and define $\bm{1}_n$ to be the $n\times 1$ matrix all of whose
entries are equal to 1. Define $C$ to be the $n\times n$ matrix
\[
C=\begin{pmatrix}
& & & 1 & -1 \\ & & \iddots & -1 & 0 \\ & 1 & \iddots & & \vdots \\ 1 & -1 & & & 0 \\ -1 & & & & 1
\end{pmatrix} \in M_n(\bm{Z}).
\]

For examples, when $n=5$, $C$ is equal to
\[
\begin{pmatrix}
0 & 0 & 0 & 1 & -1 \\ 0 & 0 & 1 & -1 & 0 \\ 0 & 1 & -1 & 0 & 0 \\ 1 & -1 & 0 & 0 & 0 \\ -1 & 0 & 0 & 0 & 1
\end{pmatrix}.
\]

\medskip
Step 2.
Write the actions of $\sigma$ and $\tau$ on $R^{ab}$ with respect to the ordered basis $\bar{b}_1,\ldots,\bar{b}_{n-1},\bar{b}_0,\bar{c}_1,\ldots,\bar{c}_{n-1},\bar{c}_0,\bar{a}$.
It is easy to verify that
\[
\sigma\mapsto \left(\begin{array}{@{}cc;{3pt/2pt}c@{}}
A & 0 & \\ 0 & A & \\ \hdashline[3pt/2pt] & & 1
\end{array}\right)\in M_{2n+1}(\bm{Z}), \quad
\tau\mapsto \left(\begin{array}{@{}cc;{3pt/2pt}c@{}} AB & 0 &
\bm{1}_n
\\ C & B & \\ \hdashline[3pt/2pt] & & -1
\end{array}\right)\in M_{2n+1}(\bm{Z}).
\]

It is easy to check that $C=B-AB$.

We will change the basis of $R^{ab}$ in the following step so that the actions of $\sigma$ and $\tau$ correspond to the matrices
\[
\left(\begin{array}{@{}c;{3pt/2pt}cc@{}}
A & & \\ \hdashline[3pt/2pt] & A & 0 \\ & 0 & 1
\end{array}\right)\quad \mbox{and} \quad
\left(\begin{array}{@{}c;{3pt/2pt}cc@{}} B & & \\
\hdashline[3pt/2pt] & B & \bm{1}_n \\ & 0 & -1
\end{array}\right).
\]

Compare with the matrix forms of $M_+$ and $\widetilde{M}_+$ in Definition \ref{d3.1} and Definition \ref{d3.3}.
We find that $R^{ab}\simeq M_+\oplus \widetilde{M}_+$.

\medskip
Step 3.
Remember $A^n=B^2=(AB)^2=I_n$.
Thus $BA=A^{-1}B$ and $A^{-\frac{n-1}{2}}=A^{\frac{n+1}{2}}$.

Define
\[
P_1=\left(\begin{array}{@{}cc;{3pt/2pt}c@{}}
I_n & 0 & \\ I_n-A^{\frac{n-1}{2}} & I_n & \\ \hdashline[3pt/2pt] & & 1
\end{array}\right).
\]

Then we find
\begin{gather*}
P_1\cdot \left(\begin{array}{@{}cc;{3pt/2pt}c@{}}
A & 0 & \\ 0 & A & \\ \hdashline[3pt/2pt] & & 1
\end{array}\right) \cdot P_1^{-1}=\left(\begin{array}{@{}cc;{3pt/2pt}c@{}}
A & 0 & \\ 0 & A & \\ \hdashline[3pt/2pt] & & 1
\end{array}\right), \\
P_1\cdot \left(\begin{array}{@{}cc;{3pt/2pt}c@{}} AB & 0 &
\bm{1}_n
\\ C & B & \\ \hdashline[3pt/2pt] & & -1
\end{array}\right)\cdot P_1^{-1}=\left(\begin{array}{@{}cc;{3pt/2pt}c@{}}
AB & 0 & \bm{1}_n \\ 0 & B & \\ \hdashline[3pt/2pt] & & -1
\end{array}\right).
\end{gather*}

\medskip
Define
\[
P_2=\left(\begin{array}{@{}cc;{3pt/2pt}c@{}}
A^{\frac{n-1}{2}} & 0 & \\ 0 & I_n & \\ \hdashline[3pt/2pt] & & 1
\end{array}\right).
\]

Then we find
\begin{gather*}
P_2\cdot \left(\begin{array}{@{}cc;{3pt/2pt}c@{}}
A & 0 & \\ 0 & A & \\ \hdashline[3pt/2pt] & & 1
\end{array}\right) \cdot P_2^{-1}=\left(\begin{array}{@{}cc;{3pt/2pt}c@{}}
A & 0 & \\ 0 & A & \\ \hdashline[3pt/2pt] & & 1
\end{array}\right), \\
P_2\cdot \left(\begin{array}{@{}cc;{3pt/2pt}c@{}} AB & 0 &
\bm{1}_n
\\ 0 & B & \\ \hdashline[3pt/2pt] & & -1
\end{array}\right)\cdot P_2^{-1}=\left(\begin{array}{@{}cc;{3pt/2pt}c@{}}
B & 0 & \bm{1}_n \\ 0 & B & \\ \hdashline[3pt/2pt] & & -1
\end{array}\right).
\end{gather*}

\medskip
Define
\[
P_3=\left(\begin{array}{@{}cc;{3pt/2pt}c@{}}
0 & I_n & \\ I_n & 0 & \\ \hdashline[3pt/2pt] & & 1
\end{array}\right).
\]

We finally get
\begin{gather*}
P_3\cdot \left(\begin{array}{@{}cc;{3pt/2pt}c@{}}
A & 0 & \\ 0 & A & \\ \hdashline[3pt/2pt] & & 1
\end{array}\right) \cdot P_3^{-1}=\left(\begin{array}{@{}cc;{3pt/2pt}c@{}}
A & 0 & \\ 0 & A & \\ \hdashline[3pt/2pt] & & 1
\end{array}\right), \\
P_3\cdot \left(\begin{array}{@{}cc;{3pt/2pt}c@{}} B & 0 & \bm{1}_n
\\ 0 & B & \\ \hdashline[3pt/2pt] & & -1
\end{array}\right)\cdot P_3^{-1}=\left(\begin{array}{@{}cc;{3pt/2pt}c@{}}
B & 0 & \\ 0 & B & \bm{1}_n \\ \hdashline[3pt/2pt] & & -1
\end{array}\right).
\end{gather*}

\bigskip
The second proof --------------

Step 1. Let the notations be the same as in Definition \ref{d1.1}. Define $H:= \langle \tau \rangle$ to be the subgroup of $G \simeq D_n$ generated by $\tau$. The augmentation map $\bm{Z}[H] \to \bm{Z}$ induces a surjective $G$-morphism $\bm{Z}[G] \to \bm{Z}[G/H]$ since $\bm{Z}[G]\simeq \fn{Ind}^G_H \bm{Z}[H]$ and $\bm{Z}[G/H]\simeq \fn{Ind}^G_H \bm{Z}$. Let $I_G$ and $I_{G/H}$ be the augmentation $G$-lattices associated to $\bm{Z}[G] \to \bm{Z}$ and $\bm{Z}[G/H] \to \bm{Z}$ respectively. From the Snake Lemma, there is a short exact sequence of $G$-lattices
\begin{align*}
0\to G\cdot I_H\to I_G\xrightarrow{\pi} I_{G/H}\to 0
\end{align*}
where $G\cdot I_H$ denotes the kernel of the map $\bm{Z}[G] \to \bm{Z}[G/H]$, which is nothing but the $G$-sublattice of $I_G$ generated by $1 - \tau$.

Note that the above exact sequence is the exact sequence in Step 1 of the proof of \cite[Theorem 8.5: (iii) $\Rightarrow$ (i)]{Gr1} (see the last line in page 57 there). By \cite[page 58, Step 2]{Gr1}, this exact sequence splits.

\bigskip
Step 2. Let $1\to R\to F\xrightarrow{\varepsilon} G\to 1$ be the free
presentation of $D_n$ where $F=\langle s_1,s_2\rangle$ the free
group of rank 2, and $\varepsilon(s_1)=\sigma$,
$\varepsilon(s_2)=\tau$.

Let $I_F$ and $I_G$ be the augmentation ideals of $\bm{Z}[F]$ and  $\bm{Z}[G]$ respectively. By \cite[page 199, Corollary 6.4]{HS}, we have a short exact sequence of $G$-lattices
\begin{align*}
0\to R^{ab}\to \bm{Z}[G]\otimes_{\bm{Z}[F]} I_F\xrightarrow{\nu}I_G\to 0.
\end{align*}

Since $I_F$ is a free $\bm{Z}[F]$-module of rank two by \cite[page 196, Theorem 5.5]{HS},
we find that $\bm{Z}[G]\otimes_{\bm{Z}[F]} I_F \simeq \bm{Z}[G]^{(2)}$. Explicitly, replacing the free generators $1-s_1$ and $1-s_2$ of the $\bm{Z}[F]$-module $I_F$ by $s_1^{(n-1)/2}(1-s_1)$ and $1-s_2$. We find that
\begin{align*}
\bm{Z}[G]\otimes_{\bm{Z}[F]} I_F=\bm{Z}[G](1\otimes(1-s_2))\oplus
\bm{Z}[G](1\otimes s_1^{\frac{n-1}{2}}(1-s_1))\simeq\bm{Z}[G]^{(2)}.
\end{align*}

Note that $\nu(1 \otimes (s_1^{(n-1)/2}(1-s_1)))=\sigma^{(n-1)/2}(1-\sigma)$ and $\nu(1\otimes (1- s_2))=1-\tau$. Thus $\nu(\bm{Z}[G]\otimes (1- s_2))=G \cdot I_H \simeq M_{-}$.

 \bigskip
Step 3. The morphism $\nu$ induces the morphisms $\nu_1$ and $\nu_2$ in the following commutative diagram
\begin{align*}
\xymatrix{
0 \ar[r]
& \bm{Z}[G](1\otimes (1-s_2)) \ar@{->>}[d]^{\nu_1} \ar[r]
& \bm{Z}[G]\otimes_{\bm{Z}[F]} I_F \ar@{->>}[d]^{\nu}\ar[r]
& \bm{Z}[G](1\otimes s_1^{\frac{n-1}{2}}(1-s_1))\ar@{->>}[d]^{\nu_2}\ar[r]
& 0\\
0 \ar[r]
& G\cdot I_H \ar[r]
& I_G \ar[r]
& I_{G/H} \ar[r]
& 0
}
\end{align*}
where the first row is exact (obvious !), and the second row is exact by Step 1. We will show that the columns are surjective in the next step.

\bigskip
Step 4. We will apply the Snake Lemma to the commutative diagram in Step 3.

Note that $\fn{Ker}(\nu)=R^{ab}$ by Part (1) of Lemma \ref{l5.2}.

Since $\nu_1:\bm{Z}[G](1\otimes (1-s_2))\to G \cdot I_G$ sends the generator $1\otimes (1-s_2)$ to $1-\tau$, $\nu_1$ is surjective and $\fn{Ker}(\nu_1)$ is isomorphic to the $G$-sublattice $\sum_{0\le i\le n-1} \bm{Z}(\sigma^i+\tau\sigma^{n-i})$ of $\bm{Z}[G]$. Note that this sublattice is nothing but $M_{+}$ (see Definition \ref{d3.1} and the matrix presentation of $M_{+}$).

On the other hand, the map $\nu_2: \bm{Z}[G] (1 \otimes (s_1^{(n-1)/2}(1-s_1))) \to I_{G/H}$ sends the generator $1\otimes(s_1^{(n-1)/2}(1-s_1))$ to $\sigma^{(n-1)/2} - \sigma^{(n+1)/2} \in I_{G/H}$. Thus it is also surjective and $\fn{Ker}(\nu_2)$ is isomorphic to the $G$-sublattice $\bm{Z}[G](1+\tau)+\bm{Z}(\sum_{0\le i \le n-1}\tau\sigma^i)$ of $\bm{Z}[G]$, which is isomorphic to $\widetilde{M}_{+}$ by Definition \ref{d3.3}.

{}From the Snake Lemma, we find an exact sequence of $G$-lattices
\begin{align*}
0\to M_+\to R^{ab}\to \widetilde{M}_+\to 0.
\end{align*}

The above sequence splits by \cite[Proposition 1.2]{Len}, because $M_+$ is a permutation lattice by definition and $\widetilde{M}_+$ is an invertible lattice by Theorem \ref{t3.4}. Alternatively, since $\widetilde{M}_+$ is an invertible lattice, it is a direct summand of some permutation lattice $P=\oplus_{H^{\prime}}\bm{Z}[G/H^{\prime}]$. It follows that $\fn{Ext}_G^1(\widetilde{M}_+, M_+)$ is a direct summand of $\oplus_{H^{\prime}}\fn{Ext}_G^1(\bm{Z}[G/H^{\prime}], M_+) \simeq \oplus_{H^{\prime}}H^1(H^{\prime}, M_+)=0$ because $M_+$ is a permutation lattice.
\qedhere
\end{proof}

\section{Rationality problems}

Let $k\subset L$ be a field extension.
Recall the definition that $L$ is $k$-rational (resp.\ stably $k$-rational, retract $k$-rational) in Definition \ref{d2.4}.

\begin{defn} \label{d6.1}
Let $G$ be any finite group, $k$ be any field. Let $k(x_g:g\in G)$
be the rational function field in $|G|$ variables over $k$ with a
$G$-action via $k$-automorphism defined by $h\cdot x_g=x_{hg}$ for
any $h,g\in G$. Define $k(G):= k(x_g: g\in G)^G$ the fixed field.
Noether's problem asks whether $k(G)$ is $k$-rational \cite{Sw3}.
\end{defn}

The following theorem is called the No-Name Lemma by some authors.

\begin{theorem}[{\cite[Theorem 2.1]{CHK}}] \label{t6.2}
Let $L$ be a field and $G$ be a finite group acting on $L(x_1,\ldots,x_m)$,
the rational function field of $m$ variables over $L$.
Suppose that

{\rm (i)}
for any $\sigma \in G$, $\sigma(L)\subset L$;

{\rm (ii)}
the restriction of the action of $G$ to $L$ is faithful;

{\rm (iii)}
for any $\sigma\in G$,
\[
\begin{pmatrix} \sigma(x_1) \\ \vdots \\ \sigma(x_m) \end{pmatrix}
=A(\sigma)\begin{pmatrix} x_1 \\ \vdots \\ x_m
\end{pmatrix}+B(\sigma)
\]
where $A(\sigma)\in GL_m(L)$ and $B(\sigma)$ is an $m\times 1$ matrix over $L$.
Then $L(x_1,\ldots,x_m)=L(z_1,\ldots,z_m)$ where $\sigma(z_i)=z_i$ for any $\sigma\in G$, any $1\le i\le m$.
In particular, $L(x_1,\ldots,x_m)^G=L^G(z_1,\ldots,z_m)$.
\end{theorem}

The following proposition is essentially the same as \cite[page 30, Proposition 1.4]{Len} and \cite[page 35, Lemma 10.2]{Sw3}. We modify the formulation slightly for the use of later application.

\begin{prop} \label{p6.3}
Let $G$ be a finite group, $M$ be a $G$-lattice. Let $k'/k$ be a
finite Galois extension such that there is a surjection $G\to
\fn{Gal}(k'/k)$. Suppose that there is an exact sequence of
$G$-lattices $0\to M_0\to M\to Q\to 0$ where $Q$ is a permutation
$G$-lattice. If $G$ is faithful on the field $k'(M_0)$, then
$k'(M)=k'(M_0)(x_1,\ldots,x_m)$ for some elements
$x_1,x_2,\ldots,x_m$ satisfying $m=\fn{rank}_{\bm{Z}} Q$,
$\sigma(x_j)=x_j$ for any $\sigma\in G$, any $1\le j\le m$.
\end{prop}

\begin{proof}
Note that the action of $G$ on $k'(M)$ is the purely
quasi-monomial action in Definition \ref{d2.1}.

Write $M_0=\bigoplus_{1\le i\le n} \bm{Z}\cdot u_i$,
$Q=\bigoplus_{1\le j\le m} \bm{Z}\cdot v_j$. Choose elements
$w_1,\ldots,w_m\in M$ such that $w_j$ is a preimage of $v_j$ for
$1\le j\le m$. It follows that $\{u_1,\ldots,
u_n,w_1,\ldots,w_m\}$ is a $\bm{Z}$-basis of $M$.

For each $\sigma\in G$, since $Q$ is permutation,
$\sigma(w_j)-w_l\in M_0$ for some $w_l$ (depending on $j$). In the
field $k'(M)$, if we write
$k'(M)=k(u_1,\ldots,u_n,w_1,\ldots,w_m)$ as the rational function
field in $m+n$ variables over $k'$, then
$\sigma(w_j)=\alpha_j(\sigma)w_l$ for some $\alpha_j(\sigma)\in
k'(M_0)$.

Since $G$ is faithful on $k'(M_0)$, apply Theorem \ref{t6.2}.
\end{proof}

In the following theorem we consider the rationality of $k(R^{ab})^{D_n}$ and $K(R^{ab})^{D_n}$ ($n$ is odd). Since the proofs are almost the same, we formulate the fixed fields as $k'(R^{ab})^{D_n}$ such that there is a surjection $D_n\to \fn{Gal}(k'/k)$, that is, $k'=k$ in the former case, and $k'=K$ in the latter case.

\begin{theorem} \label{t6.4}
Let $1\to R\to F\xrightarrow{\varepsilon} D_n\to 1$ be the free
presentation of $D_n$ in Definition \ref{d1.1}. Let $k'/k$ be a
finite Galois extension such that there is a surjection $D_n\to
\fn{Gal}(k'/k)$. Assume that $n$ is an odd integer $\ge 3$. Then
$k'(R^{ab})^{D_n}\simeq k'(\bm{Z}[D_n])^{D_n}(t)$ where $t$ is an
element transcendental over $k(D_n)$.

Consequently, if $k(D_n)$ is
rational over $k$ (e.g.\ $\zeta_n+\zeta_n^{-1}\in k$ where
$\zeta_n$ is a primitive $n$-th root of unity), then
$k(R^{ab})^{D_n}$ is rational over $k$. On the other hand, if
$K/k$ is a Galois extension with $D_n\simeq \fn{Gal}(K/k)$, then
$K(R^{ab})^{D_n}$ is rational over $k$.
\end{theorem}

\begin{proof}
Step 1. First of all, we explain how the last two assertions
follow from $k'(R^{ab})^{D_n}\simeq k'(\bm{Z}[D_n])^{D_n}(t)$.

If $k=k'$, then
$k(R^{ab})^{D_n}=k'(\bm{Z}[D_n])^{D_n}(t)=k(\bm{Z}[D_n])^{D_n}(t)$.
Note that the fixed field $k(\bm{Z}[D_n])^{D_n}$ is nothing but
$k(D_n)$ in Definition \ref{d6.1}. It is known that, if
$\zeta_n+\zeta_n^{-1}\in k$, then $k(D_n)$ is $k$-rational (see
\cite[Proposition 2.6]{CHK}). Hence the result.

If $k'=K$ and $D_n\simeq \fn{Gal}(K/k)$,
then $K(R^{ab})^{D_n}=K(\bm{Z}[D_n])^{D_n} (t)$.
Since $K(\bm{Z}[D_n])=K(x_1,x_2,\ldots,x_{2n})$ by Proposition \ref{p6.3} where $g\cdot x_i=x_i$ for all $g\in D_n$ and $1\le i\le 2n$.
Thus $K(\bm{Z}[D_n])^{D_n}=K^{D_n}(x_1,\ldots,x_{2n})=k(x_1,\ldots,x_{2n})$.

\medskip
Step 2.
In this step and Step 3, we will show that $k'(R^{ab})^{D_n}=k'(\bm{Z}[D_n])^{D_n}(t)$.

By Theorem \ref{t5.7}, $k'(R^{ab})=k'(M_+\oplus\widetilde{M}_+)$.
Since $\widetilde{M}_+$ is a faithful $D_n$-lattice, the action of
$D_n$ on $k'(\widetilde{M}_+)$ is faithful. By Proposition
\ref{p6.3} (because $M_+$ is a permutation lattice), we find that
$k'(M_+\oplus\widetilde{M}_+)^{D_n}=k'(\widetilde{M}_+)^{D_n}(x_1,\ldots,$
$x_n)$.

By applying Proposition \ref{p6.3} to $k'(\widetilde{M}_+\oplus\bm{Z})$, we obtain that $k'(\widetilde{M}_+\oplus\bm{Z})^{D_n}=k'(\widetilde{M}_+)^{D_n}(t)$.

Hence $k'(R^{ab})^{D_n}=k'(\widetilde{M}_+)^{D_n}(x_1,\ldots,x_n)\simeq k'(\widetilde{M}_+)^{D_n}(t)(x_1,\ldots,x_{n-1}) \simeq k'(\widetilde{M}_+\oplus\bm{Z})^{D_n}(y_1,\ldots,y_{n-1})$.

{}From Theorem \ref{t3.4}, $\widetilde{M}_+\oplus\bm{Z}\simeq
\bm{Z}[D_n/\langle\sigma\rangle]\oplus
\bm{Z}[D_n/\langle\tau\rangle]$. Note that
$\bm{Z}[D_n/\langle\tau\rangle]$ is a faithful $D_n$-lattice.
Hence $k'(\widetilde{M}_{+}\oplus\bm{Z})^{D_n}(y_1,\ldots,y_{n-1})=k'(\bm{Z}[D_n/\langle\sigma\rangle]\oplus
\bm{Z}[D_n/\langle\tau\rangle])^{D_n}$ $(y_1,\ldots,y_{n-1})$
$=k'(\bm{Z}[D_n/\langle\tau\rangle])^{D_n}(t_1,t_2,y_1,\ldots,y_{n-1})$ by Proposition \ref{p6.3} again.

In conclusion, $k'(R^{ab})^{D_n}\simeq
k'(\bm{Z}[D_n/\langle\tau\rangle])^{D_n}(t_1,t_2,y_1,\ldots,y_{n-1})$.

\medskip
Step 3.
Consider $k'(\bm{Z}[D_n])^{D_n}$.
Note that $k'(\bm{Z}[D_n])=k'(x(\sigma^i),x(\sigma^i\tau):0\le i\le n-1)$ where $g\cdot x(\sigma^i)=x(g\sigma^i)$,
$g\cdot x(\sigma^i\tau)=x(g\sigma^i\tau)$ for any $g\in D_n$.

Define $z_i=x(\sigma^i)+x(\sigma^i\tau)$ for $0\le i\le n-1$.
Then the actions of $\sigma$ and $\tau$ are given by
\begin{align*}
\sigma &: z_0\mapsto z_1\mapsto \cdots \mapsto z_{n-1}\mapsto z_0, \\
\tau &: z_0\mapsto z_0,~ z_i\mapsto z_{n-i} \mbox{ for }1\le i\le n-1.
\end{align*}

In other words, $k'(z_i:0\le i\le n-1)$ is $D_n$-isomorphic to $k'(\bm{Z}[D_n/\langle\tau\rangle])$.

Moreover, $\sigma\cdot x(\sigma^i)=x(\sigma^{i+1})$, $\tau\cdot x(\sigma^i)=x(\sigma^{n-i}\tau)=-x(\sigma^{n-i})+z_{n-i}$.
Thus we may apply Theorem \ref{t6.2} to $k'(\bm{Z}[D_n])^{D_n}=k'(z_i,x(\sigma^i):0\le i\le n-1)^{D_n}$.
It follows that $k'(\bm{Z}[D_n])^{D_n}=k'(z_i:0\le i\le n-1)^{D_n}(u_1,\ldots,u_n) \simeq k'(\bm{Z}[D_n/\langle\tau\rangle])^{D_n}(u_1,\ldots,u_n)$.

Obviously we have $k'(\bm{Z}[D_n])^{D_n}(t^{\prime}) \simeq k'(\bm{Z}[D_n/\langle\tau\rangle])^{D_n}(u_1,\ldots,u_n, t^{\prime})$.

Combining the above formula and the final formula in Step 2, we find that $k'(R^{ab})^{D_n}$ $\simeq k'(\bm{Z}[D_n])^{D_n}(t^{\prime})$. Done.
\end{proof}

\begin{proof}[Proof of Theorem \ref{t1.4} and Theorem \ref{t1.3}]
---------------

Apply Theorem \ref{t5.8} and Theorem \ref{t6.4}.
\end{proof}

\begin{proof} [An alternative proof of Theorem \ref{t1.5}]
---------------

The following proof is different from that given by Snider \cite{Sn}. We will show that $k(R^{ab})^{D_2}$ is rational over $k$. Our proof uses an idea of Kunyavskii \cite{Ku1}.

By Lemma \ref{l5.6}, $R^{ab}=\bm{Z}\cdot\bar{a} \oplus \bm{Z}\cdot\bar{b}_0 \oplus \bm{Z}\cdot\bar{b}_1 \oplus \bm{Z}\cdot\bar{c}_0 \oplus \bm{Z}\cdot\bar{c}_1$.

Define $M_1=\{x\in R^{ab}:(\sum_{g\in D_2} g)\cdot x=0\}$, $M_2=R^{ab}/M_1$.
Then $0\to M_1 \to R^{ab}\to M_2\to 0$ is a short exact sequence of $\bm{Z}[D_2]$-lattices.
Note that $M_2$ is a trivial $\bm{Z}[D_2]$-lattice because of the definition of $M_1$.

It is not difficult to show that $M_1=\bm{Z}\cdot (\bar{a}-\bar{b}_0)\oplus \bm{Z}\cdot (\bar{b}_0-\bar{b}_1)\oplus \bm{Z}\cdot(\bar{c}_0-\bar{c}_1)$.
Hence $\fn{rank}_{\bm{Z}} M_1=3$, $\fn{rank}_{\bm{Z}} M_2=2$.
It follows that $M_2\simeq \bm{Z}^{(2)}$.

By Proposition \ref{p6.3} $k(R^{ab})^{D_2}=k(M_1)^{D_2}(x_1,x_2)$.
On the other hand, $D_2$ acts on $k(M_1)$ by purely monomial $k$-automorphisms in the sense of \cite{HK}.
By \cite{HK}, $k(M_1)^{D_2}$ is $k$-rational.
\end{proof}

\begin{remark} (1) If $R_0^{ab}$ is the relation module associated to a free presentation $1\to R\to F\xrightarrow{\varepsilon} D_n\to 1$ where $F$ is a free group of rank $d$. Let $R^{ab}$ be the relation module in Definition \ref{d1.1}. By Lemma \ref{l5.3}, $R_0^{ab}\oplus \bm{Z}[D_n]^{(2)}=R^{ab} \oplus \bm{Z}[D_n]^{(d)}$. Thus the rationality problem for $R_0^{ab}$ may be reduced to that that for $R^{ab}$.

(2) For any field $k$, there are field extensions $k \subset K \subset
L$ such that $K$ is $k$-rational, $L$ is not retract $K$-rational,
but $L$ is still $k$-rational.

In fact, take the group $D_2$ and consider $R^{ab}$. Define
$L=k(R^{ab} \oplus \bm{Z}[D_2])^{D_2}$, $K=k(\bm{Z}[D_2])^{D_2}$.
Since the group $D_2$ acts faithfully on $k(R^{ab})$, we may apply
Theorem \ref{t6.2} to the field $k(R^{ab} \oplus \bm{Z}[D_2])$.
Thus $L=k(R^{ab} \oplus \bm{Z}[D_2])^{D_2} =k(R^{ab})^{D_2}(t_1,
\ldots,t_4)$. Since $k(R^{ab})^{D_2}$ is $k$-rational by Theorem
\ref{t1.5}. Thus $L$ is $k$-rational. Clearly
$K=k(\bm{Z}[D_2])^{D_2}$ is $k$-rational, a well-known result in
Noether's problem.

On the other hand, $L=k(R^{ab} \oplus \bm{Z}[D_2])^{D_2}$ may be
regarded as a function field of an algebraic torus over
$K=k(\bm{Z}[D_2])^{D_2}$, which is split by a $D_2$-extension with
character module $R^{ab}$. By Theorem \ref{t5.8} $L$ is not
retract $K$-rational.
\end{remark}

\begin{theorem} \label{t6.9}
Let $k$ be any field and $M$ be a $D_2$-lattice with rank $\le 5$, then $k(M)^{D_2}$ is stably rational over $k$
\end{theorem}

\begin{proof}
We use the idea of Kunyavskii \cite{Ku1} again.

Let $M$ be a $D_2$-lattice with rank $\le 5$. Define $M_1=\{x\in M:(\sum_{g\in D_2} g)\cdot x=0\}$, $M_2=M/M_1$.
Then $0\to M_1 \to M\to M_2\to 0$ is a short exact sequence of $\bm{Z}[D_2]$-lattices.

Note that $M_2$ is a trivial $\bm{Z}[D_2]$-lattice because of the definition of $M_1$. By Proposition \ref{p6.3}, it suffices to show that $k(M_1)^{D_2}$ is stably rational over $k$. A $D_2$-lattice $N$ satisfying that $(\sum_{g\in D_2} g)\cdot N=0$ is called an anisotropic lattice.

\medskip
Although there are infinitely many indecomposable lattices over $D_2 \simeq C_2 \times C_2$, the total number of the indecomposable anisotropic $D_2$-lattices is finite. In fact, there are only $8$ such lattices by Drozd \cite{Dr} (also see \cite[page 538; Vo, page 57]{Ku1}).

Only two of them are of rank $3$; the remaining ones are of lower rank. The rank $3$ lattices are $I_{D_2}$ and $I_{D_2}^0$ where $I_{D_2}$ is the augmentation ideal of $\bm{Z}[D_2]$ (see Lemma \ref{l5.2}). By Theorem \ref{t2.7}, $[I_{D_2}^0]^{fl} \neq 0$. It is not difficult to verify that all the other indecomposable anisotropic lattices $N$ satisfy that $[N]^{fl}=0$.

\bigskip
Since $M_1$ is of rank $\le 5$, it is a direct sum of these indecomposable anisotropic lattices such that the lattice $I_{D_2}^0$ appears at most once. If $N$ is an indecomposable anisotropic lattice other than $I_{D_2}^0$, choose a flabby resolution $0 \rightarrow N \rightarrow P_1 \rightarrow P_2 \rightarrow 0$ where $P_1$ and $P_2$ are permutation lattices. By Proposition \ref{p6.3}, $k(N)^{D_2}$ is stably isomorphic to $k(P_1)^{D_2}$ (if $N$ is not a faithful $D_2$-lattice, but is a faithful $D_2/H$-lattice, we may assume that $P_1$ and $P_2$ are permutation $D_2/H$-lattices by \cite[page 179-180]{CTS}).

\medskip
In summary, $k(M)^{D_2}$ is stably isomorphic to $k(P)^{D_2}$ or $k(I_{D_2}^0 \oplus Q)^{D_2}$ where $P$ and $Q$ are permutation lattices.

Since the action of $D_2$ on $k(P)$ is linear, it is easy to show that $k(P)^{D_2}$ is rational over $k$. Applying Theorem \ref{t6.2} to $k(I_{D_2}^0 \oplus Q)^{D_2}$, we find that $k(I_{D_2}^0 \oplus Q)^{D_2}$ is rational over $k(I_{D_2}^0)^{D_2}$. But $k(I_{D_2}^0)^{D_2}$ is $k$-rational by \cite{HK}. Hence $k(I_{D_2}^0 \oplus Q)^{D_2}$ is $k$-rational.
\end{proof}

The method in the proof of Theorem \ref{t6.4} may be applied to other rationality problems.
We record one of them in the following.

\begin{theorem} \label{t6.7}
Let $k$ be any field, $G=\langle \sigma,\tau:\sigma^n=\tau^2=1,\tau\sigma\tau^{-1}=\sigma^{-1}\rangle\simeq D_n$ where $n\ge 3$ is an odd integer.
Define an action of $G$ on the rational function field $k(x_1,x_2,\ldots,x_{n-1})$ through $k$-automorphisms defined by
\begin{align*}
\sigma &: x_1\mapsto x_2\mapsto \cdots \mapsto x_{n-1} \mapsto 1/(x_1x_2\cdots x_{n-1}), \\
\tau &: x_i \leftrightarrow x_{n-i}.
\end{align*}

Then $k(x_1,\ldots, x_{n-1})^G$ is stably $k$-rational if and only if $k(G)$ is stably $k$-rational.
\end{theorem}

\begin{proof}
We may write $k(x_1,\ldots,x_{n-1})^G=k(M)^G$ where $M$ is the $G$-lattice $\bm{Z}[G/\langle\tau\rangle]$.
Note that $M$ is nothing but $N_+$ in Definition \ref{d3.2}.

By Lemma \ref{l4.5}, we find that $0\to N_+\to \widetilde{M}_-\to \bm{Z}[G/\langle\sigma\rangle]\to 0$ is an exact sequence of $G$-lattices.
By Proposition \ref{p6.3}, $k(\widetilde{M}_-)$ is $G$-isomorphic to $k(M)(y_1,y_2)$ where $\sigma(y_i)=\tau(y_i)=y_i$ for $1\le i\le 2$.

By Theorem \ref{t3.5},
$\widetilde{M}_-\oplus\bm{Z}[G/\langle\tau\rangle]\simeq
\bm{Z}[G]\oplus \bm{Z}$. Thus, by Proposition \ref{p6.3} again,
$k(\widetilde{M}_-)(z_1,\ldots,z_n)$ is $G$-isomorphic to
$k(\bm{Z}[G])(z_0)$ where $\sigma(z_i)=\tau(z_i)=z_i$ for $0\le
i\le n$.

Hence $k(M)^G$ is stably isomorphic to $k(\bm{Z}[G])^G$, which is nothing but $k(G)$.
\end{proof}

\begin{remark}
We thank the referee who recast the above result as follows: Let $G$ be a finite group and $M$ be a faithful $G$-lattice satisfying that $[M]^{fl}$ is permutation (see Definition \ref{d2.1}). For any field $k$, $k(M)^G$ is stably rational over $k$ if and only if so is $k(G)$.

The proof goes as follows: Considering the fixed field $k(M \oplus \bm{Z}[G])^G$. Note that $k(M \oplus \bm{Z}[G])^G$ is rational over $k(M)^G$ by Proposition \ref{p6.3}. On the other hand, since $[M]^{fl}$ is permutation, $k(M \oplus \bm{Z}[G])^G$ is stably rational over $k(\bm{Z}[G])^G = k(G)$ by Theorem \ref{t2.5}. Hence $k(M)^G$ and $k(G)$ are stably isomorphic.

It remains to show that the lattice $M$ in the proof of Theorem \ref{t6.7} satisfies the condition that $[M]^{fl}$ is permutation. In fact, from the proof, $M \simeq N_{+}$. Moreover, we have the exact sequences $0\to N_+\to \widetilde{M}_-\to \bm{Z}[G/\langle\sigma\rangle]\to 0$ (by Lemma \ref{l4.5}) and
$\widetilde{M}_-\oplus\bm{Z}[G/\langle\tau\rangle]\simeq
\bm{Z}[G]\oplus \bm{Z}$ (by Theorem \ref{t3.5}). Thus $[M]^{fl}=0$.

Here is another interpretation of the lattice $M$. Let $n\ge 3$ be an odd integer and $G=\langle \sigma, \tau: \sigma^n=\tau^{2^d}=1, \tau \sigma \tau^{-1}=\sigma^r \rangle$ where $r^2 \equiv 1$ (mod $n$). Define $H=\langle \tau \rangle$ and $I_{G/H}$ the augmentation lattice of the augmentation map $\bm{Z}[G/H] \to \bm{Z}$. Then we get an exact sequence $0 \to I_{G/H} \to \bm{Z}[G/H] \to \bm{Z} \to 0$. Taking the dual of it, we find that $0 \to \bm{Z} \to \bm{Z}[G/H] \to I_{G/H}^0 \to 0$. It is easy to see that, when $G$ is the dihedral group in Theorem \ref{t6.7}, the lattice $M$ in the proof of Theorem \ref{t6.7} is isomorphic to $I_{G/H}^0$. When $n=p^e$ for some prime number $p$, it is shown by Colliot-Th\'el\`ene and Sansuc that $[I_{G/H}^0]^{fl}=0$ (see, for examples, \cite[page 55, Theorem 4]{Vo}).

\end{remark}

\bigskip

\bigskip

\renewcommand{\refname}{\centering{References}}

\end{document}